%% file: semigroupcstar.tex
\documentclass[a4paper, english, 11pt]{amsart}

\usepackage[T1]{fontenc}
\usepackage{babel}
\usepackage{amsmath, amsfonts, amssymb, amsthm} 
\usepackage{url}      
\usepackage[numbers]{natbib}   
\usepackage{mathrsfs}

\numberwithin{equation}{section}

\newtheoremstyle{thm}{10pt}{5pt}{}{}{\bfseries}{}{ }{\thmname{#1}
  \thmnumber{#2} \thmnote{(#3)}}

\theoremstyle{thm}
\newtheorem{theorem}{Theorem}[subsection]
\newtheorem{definition}[theorem]{Definition}

\newtheorem{lemma}[theorem]{Lemma}
\newtheorem{corollary}[theorem]{Corollary}
\newtheorem{proposition}[theorem]{Proposition}

\newtheorem{remark}[theorem]{Remark}
\theoremstyle{remark}

\newcommand{\iprod}[1]{\langle #1\rangle}
\newcommand{\norm}[1]{\|#1\|}

\renewcommand{\chi}{E}

\DeclareMathOperator{\dom}{dom}
\DeclareMathOperator{\ran}{ran}
\DeclareMathOperator{\Piso}{Piso}
\DeclareMathOperator{\spn}{span}

\begin{document}

\title[$C^\ast$-algebras of left cancellative semigroups]{Inverse semigroup $C^\ast$-algebras associated with left cancellative semigroups}
\author{Magnus Dahler Norling}
\email{magnus.dahler.norling@gmail.com}

\maketitle

\begin{abstract}
\input{abstract.tex}
\end{abstract}

\tableofcontents

\section{Introduction}
\input{intro.tex}

\section{Semigroups}

\input{semigroup.tex}

\input{inversesemi.tex}
\input{semilattice.tex}

\section{$C^\ast$-theory}

\input{inversesemirep.tex}
\input{semigrouprep.tex}
\input{lialgebra.tex}
\input{thicksemi.tex}
\input{weakcontainment.tex}

\bibliographystyle{plain}
\bibliography{bibliography}

\end{document}

%% file: abstract.tex
To each discrete left cancellative semigroup $S$ one may associate a certain inverse semigroup $I_l(S)$, often called the left inverse hull of $S$. We show how the full and the reduced $C^\ast$-algebras of $I_l(S)$ are related to the full and reduced semigroup $C^\ast$-algebras for $S$ recently introduced by Xin Li, and give conditions ensuring that these algebras are isomorphic. Our picture provides an enhanced understanding of Li's algebras.

%% file: intro.tex
In \cite{li11}, Xin Li proposed a construction for the full $C^\ast$-algebra $C^\ast(S)$ of a discrete left cancellative semigroup $S$. For a semigroup $S$ that embeds into a group he also constructs a related $C^\ast$-algebra called $C^\ast_s(S)$. The reason one considers left cancellative semigroups is that these are the semigroups that can be faithfully represented as semigroups of isometries on Hilbert spaces. For instance, one can represent $S$ on $\ell^2(S)$ by isometries in this case. This representation is called the left regular representation of $S$ and generates what is called the Toeplitz algebra or reduced $C^\ast$-algebra of $S$, denoted $C^\ast_r(S)$. (One could of course consider right cancellative semigroups instead).

Murphy had previously constructed $C^\ast$-algebras of left cancellative semigroups, but these turned out to be very large. For instance, his $C^\ast$-algebra associated to $(\mathbb{Z}^+)^2$ is non-nuclear \cite{murphy96}. (See Li's article for more references). Li adds a few extra restrictions that make the algebras behave better. Especially he shows that $C^\ast(S)$ generalizes two important types of $C^\ast$-algebras: Nica's $C^\ast$-algebras for quasilattice ordered groups from \cite{nica92}, and the Toeplitz algebras associated with the ring of integers in a number field \cite{cuntz_deninger_laca11}.

Li also shows that a cancellative left reversible\footnote{A semigroup $S$ is left reversible if for any $s,t\in S$, $sS\cap tS\neq\varnothing$. This is also called the Ore condition.} semigroup $S$ is left amenable if and only if $C^\ast_s(S)$ and $C^\ast_r(S)$ are canonically isomorphic, but only given that the constructible right ideals of $S$ satisfy a certain technical requirement called independence. Note that Li's proof uses that left reversibility of a cancellative semigroup $S$ implies that $S$ embeds into a group and that there exists a character on $C^\ast_s(S)$.

Let $I(S)$ be the inverse semigroup of all partial bijections on $S$. For each $s\in S$, let $\lambda_s:S\to sS$ be given by $\lambda_s(t)=st$. Since $S$ is left cancellative, each $\lambda_s$ is a bijection. The set $\{\lambda_s\}_{s\in S}$ generates an inverse subsemigroup $I_l(S)\subset I(S)$ called the left inverse hull of $S$. We show that $I_l(S)$ is isomorphic to an inverse semigroup $V(S)$ of partial isometries generating $C^\ast_r(S)$. By considering the full and reduced $C^\ast$-algebras of $I_l(S)$ as for instance defined in Paterson's book \cite{paterson99} we get surjective $\ast$-homomorphisms
\[
\begin{array}{ccccccc}
C^\ast(S) &\xrightarrow{\eta}& C^\ast_0(I_l(S)) & \xrightarrow{\Lambda_0} & C^\ast_{r,0}(I_l(S)) &\xrightarrow{h} &C^\ast_r(S)
\end{array}
\]
The composition of these is the canonical $\ast$-homomorphism $C^\ast(S)\to C^\ast_r(S)$. The question of whether this is an isomorphism splits into three separate problems. When $S$ embeds into a group, we get the decomposition
\[
\begin{array}{ccccccccc}
C^\ast(S) &\xrightarrow{\pi_s}& C^\ast_s(S) & \xrightarrow{\simeq}& C^\ast_0(I_l(S)) & \xrightarrow{\Lambda_0} & C^\ast_{r,0}(I_l(S)) &\xrightarrow{h} &C^\ast_r(S)
\end{array}
\]
In particular, $C^\ast_s(S)$ and $C^\ast_0(I_l(S))$ are canonically isomorphic. 

A semigroup $S$ is said to satisfy Clifford's condition if for all $s,t\in S$, either $sS\cap tS=\varnothing$ or $sS\cap tS=rS$ for some $r\in S$. Any semigroup that is the positive cone in one of Nica's quasilattice ordered groups satisfies Clifford's condition. The $ax+b$-semigroup over an integral domain $R$ satisfies Clifford's condition if and only if every pair of elements in $R$ has a least common multiple. If $S$ satisfies Clifford's condition, $\eta$ is an isomorphism and the constructible right ideals of $S$ are independent. If $S$ is cancellative and satisfies Clifford's condition, or if $S$ embeds into a group and the constructible right ideals of $S$ are independent, then $h$ is an isomorphism.

Using Milan's work \cite{milan10} on weak containment for inverse semigroups we show that when $S$ embeds into a group $G$, $\Lambda_0$ is an isomorphism if and only if a certain Fell bundle over $G$ associated to $I_l(S)$ is amenable. In the special case when $S$ is left reversible, $\Lambda_0$ is an isomorphism if and only if $S$ is left amenable if and only if $C^\ast_0(I_l(S))$ is nuclear.

In the first part of the article we recall the algebraic theory of semigroups and inverse semigroups, and also look at an algebraic partial order and see how it is related to Nica's quasilattice ordered groups. We show that many of the properties of the positive cone in these groups can be defined in a more general context and remark that the algebraic order is not essential for the theory to work.

In the second part, we introduce the $C^\ast$-algebras associated to $S$ and $I_l(S)$, and prove the above stated results. In addition, we show that our construction generalizes a method used by Nica in \cite{nica94} to construct the $C^\ast$-algebra of a quasilattice ordered group from a certain inverse semigroup called a Toeplitz inverse semigroup.

We also prove a functoriality result for the construction $S\mapsto G(S)$ when $S$ is left reversible. Here $G(S)$ is the maximal group homomorphic image of $I_l(S)$. We use this to show that the construction $S\mapsto C^\ast(I_l(S))$ is functorial for homomorphisms into groups when $S$ is left reversible. The construction $S\mapsto G(S)$ originates from Rees' proof of Ore's Theorem: that all cancellative left reversible semigroups are group embeddable. An account of this theorem can be found in vol I, p. 35 of \cite{clifford_preston61} or in ch. 2.4 of \cite{lawson99}.

\subsection*{Acknowledgements}
We wish to thank Erik Bedos for fruitful discussions and for providing us with many helpful references. Thanks also to Mark Lawson for answering our questions about semigroups satisfying Clifford's condition.

%% file: semigroup.tex
\subsection{Semigroups and algebraic orders}

There are many sources on the algebraic theory of semigroups. See for instance \cite{clifford_preston61} or \cite{meakin11} and the references therein.

\begin{definition}
A \emph{semigroup} is a set $S$ together with a associative binary operation $\cdot:S\times S\to S$ written $(s,t)\mapsto st$ and an identity element\footnote{Usually, semigroups are not required to have identities, and semigroups with identities are called \emph{monoids}. We will however only talk about monoids in this article, and we prefer to call them semigroups.} $1\in S$. That is, for all $s,r,t\in S$, $s(rt)=(sr)t$ and $1s=s1=s$. Sometimes we write $1=1_S$.
\end{definition}

If $S$ has an element $z\in S$ such that $zs=sz=z$ for all $s\in S$, we will write $z=0=0_S$. If $S$ is a semigroup, define $S^0=S$ if $S$ already has a $0$ element, and otherwise let $S^0$ be the semigroup $S\cup\{0\}$ with extended multiplication rule $s0=0s=0$ for all $s\in S^0$.

This choice of notation can be confusing for instance in the case of $(\mathbb{Z}^+,+)$ where we have $1_{\mathbb{Z}^+}=0$, and where $\mathbb{Z}^+$ does not have an element $0_{\mathbb{Z}^+}$ in the sense of the above definition, but the notation is otherwise very convenient. (In our notation, $\mathbb{Z}^+$ denotes $\{0,1,2,\ldots\}$, while $\mathbb{N}$ denotes $\{1,2,\ldots\}$).

\begin{definition}
A \emph{homomorphism} between semigroups $S$, $S'$ is a function $f:S\to S'$ such that for all $s,t\in S$, $f(st)=f(s)f(t)$ and $f(1_S)=1_{S'}$. The homomorphism $f$ is a \emph{$0$-homomorphism} if in addition $f(0_S)=0_{S'}$ (and this term is only defined for homomorphisms between semigroups with zeroes).
\end{definition}

\begin{definition}
A semigroup $S$ is \emph{left cancellative} if for every $s,r,t\in S$, $sr=st$ implies $r=t$. Equivalently, for every $s\in S$, the map $\lambda_s:S\to sS$ given by $\lambda_s(t)= st$ is bijective. In a left cancellative semigroup, if $ss'=1$, then $ss's=1s=s1$, so $s's=1$, that is every element with a left (or right) inverse is invertible. One can similarly define right cancellativity. $S$ is \emph{cancellative} if it is both left and right cancellative.
\end{definition}

\begin{definition}
A \emph{congruence} on a semigroup $S$ is an equivalence relation $\sim$ such that for all $s,t,r,\in S$, $s\sim t$ implies $sr\sim tr$ and $rs\sim rt$. One can show that $S/\sim$ is again a semigroup and that there is a homomorphism $S\to S/\sim$ sending elements to equivalence classes. In fact, the homomorphism theorems for semigroups say that every surjective homomorphism can be constructed in this way.
\end{definition}

\begin{definition}
A subset $X\subset S$ is a \emph{right ideal} if for all $t\in X$ and $s\in S$, $ts\in X$.
\end{definition}

For $X\subset S$ and $s\in S$, define $s^{-1}(X)=\{t:st\in X\}$ and $sX=\{st:t\in X\}$. For simplicity, we will sometimes write $s^{-1}X$ for $s^{-1}(X)$. If $X\subset S$ is a right ideal, then so are $sX$ and $s^{-1}X$. The right ideals on the form $sS$ are called the \emph{principal right ideals} of $S$.

Let $\preceq$ be the relation on $S$ given by $s\preceq t$ if there exists an $r\in S$ such that $s=tr$. This relation is reflexive and transitive, so it gives a preorder on $S$. If it is antisymmetric, then it is a partial order called the \emph{algebraic order} on $S$ and we say that $S$ is \emph{algebraically ordered}. Note that $\preceq$ is often written with the opposite symbol $\succeq$ or $\geq$ (such as in Nica's work \cite{nica92}), but this is just a matter of convenience. For instance we have $5\preceq 4$ in $(\mathbb{Z}^+,+)$ with our notation.

For $s,t\in S$, $s\preceq t$ is easily seen to be equivalent to $s\in tS$ and $sS\subset tS$. It is also equivalent to $t^{-1}(\{s\})\neq\varnothing$, and if $S\subset G$ where $G$ is a group, it is equivalent to $t^{-1}s\in S$. Note that $1_S$ is a maximal element for $\preceq$ and that if $0_S$ exists it is a minimal element.

\begin{lemma}
Let $S$ be a left cancellative semigroup. Then $S$ is algebraically ordered if and only if $1$ is the only invertible element in $S$.
\end{lemma}
\begin{proof}
Suppose $rS=tS$ for some $r,t\in S$. Then there are $s,s'\in S$ such that $rs=t$ and $ts'=r$. So $ts's=t$. By left cancellation with $t$, this gives us $s's=1$. Then $s$ has a left inverse, so it is invertible since $S$ was left cancellative. If $1$ is the only invertible element in $S$, $s=1$ and this implies $r=t$.

On the other hand, suppose there are $s,s'\in S$ with $s's=1$. Then $s'sS\subset s'S\subset S=s'sS$, so if $S$ is algebraically ordered, $s'=1$.
\end{proof}

For instance when $S$ is a subsemigroup of a group $G$, $S$ is algebraically ordered if and only if $S\cap S^{-1}=\{1\}$.

%% file: inversesemi.tex
\subsection{Inverse semigroups}

Inverse semigroups are a large topic. See for instance \cite{lawson99} or \cite{paterson99} and references therein. In this section we will just give a short overview of the main concepts that we need.

\begin{definition}
A semigroup $P$ is an \emph{inverse semigroup} if for every $p\in P$, there exists a unique element $p^\ast\in P$ such that $pp^\ast p=p$ and $p^\ast pp^\ast=p^\ast$.
\end{definition}

It follows from this uniqueness property of the $\ast$-operation that for any semigroup homomorphism $f:P\to Q$ between inverse semigroups, $f(p^\ast)=f(p)^\ast$ for any $p\in P$. Let $L$ be the set of idempotents in the inverse semigroup $P$. Then $L=\{p^\ast p:p\in P\}=\{pp^\ast:p\in P\}$. One can show that $L$ is a commutative subsemigroup of $P$, so $L$ is what is called a semilattice.

\begin{definition}
A \emph{semilattice} is a commutative semigroup where every element is idempotent.
\end{definition}

\begin{lemma}
Let $L$ be a semilattice, and let $a,b\in L$. Then $a\preceq b$ if and only if $ba=a$. Hence $\preceq$ is a partial order on $L$.
\end{lemma}
\begin{proof}
If $a=ba$, then $a\in bL$, so $a\preceq b$. Suppose $a\in bL$, so $a=bc$ for some $c\in L$. Then $a=aa=bca$. So $bca=bbca=ba$, which implies $a=ba$. If $a\preceq b$ and $b\preceq a$, then $a=ba=ab=b$. So $\preceq$ is a partial order.
\end{proof}

It also follows that for $a,b\in L$, $ab$ is the greatest lower bound of $a$ and $b$. On the other hand, if $L$ is a partially ordered set where any finite subset has a unique greatest lower bound and one defines $ab$ to be the greatest lower bound of $\{a,b\}$, then $L$ is a semilattice with the product $(a,b)\mapsto ab$. We will later study partially ordered semigroups $S$, and for this it is useful to let $s\wedge t$ mean the greatest lower bound of $s$ and $t$ if it exists, while $st$ means the already existing semigroup product of $s$ and $t$. These two products only coincide if $S$ is a semilattice.

\begin{remark}
There exists a partial order defined on inverse semigroups called the natural partial order. It does not in general coincide with what we have called the algebraic order. We will not use the natural partial order explicitly in this paper.
\end{remark}

The perhaps most important example of an inverse semigroup is the semigroup $I(X)$ of all partially defined bijective maps on some set $X$. By a partially defined bijective map on $X$, we mean a bijective function $f:\dom(f)\to \ran(f)$ where $\dom(f)$ and $\ran(f)$ are subsets of $X$. The product $fg$ of $f,g\in I(X)$ is defined such that $\dom(fg)=g^{-1}(\dom(f))$ and $fg(x)=f(g(x))$ for all $x\in\dom(fg)$. Note that this product can result in the empty function, which acts as a $0$ for $I(X)$. The $\ast$-operation is given by function inversion. For any $f\in I(X)$, $f^\ast f=i_{\dom(f)}$ where $i_{\dom(f)}:\dom(f)\to\dom(f)$ is the identity map.

The Wagner-Preston Theorem states that any inverse semigroup $P$ can be faithfully represented as a subsemigroup of $I(P)$ as follows: Let $\tau:P\to I(P)$ be given such that for $p\in P$, $\dom(\tau(p))=\{q\in P:p^\ast pq=q\}$, and define $\tau(p)(q)=pq$ for all $q\in\dom(\tau(p))$.

Another important class of inverse semigroups are semigroups of partial isometries in a $C^\ast$-algebra. Note that in general the product of two partial isometries does not have to be a partial isometry. Two partial isometries can be part of the same inverse semigroup if and only if their initial and final projections commute.

The following concepts are very important in the theory of inverse semigroups.

\begin{definition}\label{def:eastunitary}
An inverse semigroup $P$ is \emph{$E$-unitary} if for every $p,q\in P$, $pq=q$ implies $p\in L$. It is \emph{$E^\ast$-unitary} (also called $0$-$E$-unitary) if for every $p,q\in P$, $pq=q$ and $q\neq 0$ implies $p\in L$.
\end{definition}

Note that if $P$ is an $E$-unitary inverse semigroup with $0$, then it is a semilattice. Note also that we can assume without loss of generality that the $q$ in either defintion is idempotent if we want to. Multiply the equation $pq=q$ on the right with $q^\ast$. This gives us $pqq^\ast=qq^\ast$ where $qq^\ast$ is idempotent. Recall that for any semigroup $S$,  $S^0=S$ if $S$ already has a $0$ element, and otherwise $S^0$ is the semigroup $S\cup\{0\}$ with extended multiplication rule $s0=0s=0$ for all $s\in S^0$.

\begin{definition}
A \emph{grading} of the inverse semigroup $P$ is a map $\varphi:P^0\to G^0$, where $G$ is a group, such that $\varphi^{-1}(\{0\})=\{0\}$ and for all $p,q\in P$, $\varphi(pq)=\varphi(p)\varphi(q)$ as long as $pq\neq 0$. $P$ is \emph{strongly $E^\ast$-unitary} if it has a grading $\varphi$ such that $\varphi^{-1}(\{1_G\})=L\setminus\{0\}$. Such a grading is sometimes said to be \emph{idempotent pure}.
\end{definition}

Note that if $\varphi:P^0\to G^0$ is a grading of $P$, $L\setminus\{0\}\subset \varphi^{-1}(\{1_G\})$ always. Note also that if $P$ is strongly $E^\ast$-unitary, then it is $E^\ast$-unitary. It turns out that if $P$ does not have a $0$, all these concepts are equivalent.

\begin{definition}\label{def:grouphomomorphicimage}
Define a relation $\sim$ on $P$ by $p\sim q$ if $pr=qr$ for some $r\in P$ (if and only if $pr=qr$ for some $r\in L$). Then $\sim$ is a congruence, and  $P/\sim$ is a group denoted $G(P)$. Let $\alpha_P:P\to G(P)$ be the quotient homomorphism. Then $P$ is $E$-unitary if and only if $\alpha_P^{-1}(1_{G(P)})=L$. $G(P)$ is often called the \emph{maximal group homomorphic image} of $P$.
\end{definition}

We will need the following lemma later:

\begin{lemma}\label{lem:gisomid}
Let $f:P\to Q$ be a surjective homomorphism between inverse semigroups. Let $L(P)$ and $L(Q)$ denote the respective semilattices of idempotents in $P$ and $Q$. Suppose the restriction of $f$ to $L(P)$ is an isomorphism onto $L(Q)$. Then $f$ is an isomorphism if and only if $f^{-1}(L(Q))=L(P)$
\end{lemma}
\begin{proof}
The only if part is trivial. Suppose $f^{-1}(L(Q))=L(P)$. Let $p,q\in P$ with $f(p)=f(q)$. Then $f(pq^\ast)=f(qq^\ast)\in L(Q)$, so by assumption $pq^\ast$ is idempotent. Since $f$ is an isomorphism restricted to $L(P)$, $pq^\ast=qq^\ast$, so $q^\ast p q^\ast=q^\ast$. Similarly, $f(q^\ast)=f(p^\ast)$, so $q^\ast p=p^\ast p$, and $pq^\ast p=p$. Thus $p=q$ by the uniqueness property for these relations in an inverse semigroup.
\end{proof}

%% file: semilattice.tex
\subsection{The semilattice $J(S)$, Clifford's condition and independence of constructible right ideals.}\label{sec:semilattice}

We will be interested in the semilattice $J(S)$ of constructible right ideals in the left cancellative semigroup $S$ given by
\[
J(S)=\left\{\bigcap_{j=1}^N t_{j1}^{-1}s_{j1}\cdots t_{jn_j}^{-1}s_{jn_j}S: N,n_j\in\mathbb{N}, s_{jk},t_{jk}\in S\right\}
\]
We will actually see in Lemma \ref{lem:simplifiedjs} that
\[
J(S)=\left\{t_{1}^{-1}s_{1}\cdots t_{n}^{-1}s_{n}S: n\in\mathbb{N}, s_{i},t_{i}\in S\right\}
\]
Here the semilattice product on $J(S)$ is given by set intersection. To motivate this study, we can reveal that $J(S)$ is isomorphic to a semilattice of projections generating the diagonal subalgebra of the $C^\ast$-algebra of the left regular representation of $S$ (also called the Toeplitz algebra of $S$ or $C^\ast_r(S)$). It is also the semilattice of idempotents in the left inverse hull of $S$. We will establish these facts later. This semilattice plays an important part in Li's theory \cite{li11}. Li's $\mathcal{J}$ is the same as our $J(S)\cup\{\varnothing\}\simeq J(S)^0$.

\begin{lemma}\label{lem:semigroupwedge}
Let $S$ be an algebraically ordered semigroup and let $s,t\in S$. If $sS\cap tS=rS$ for some $r\in S$, then $s\wedge t$ exists and equals $r$. Conversely, if $s\wedge t$ exists, then $(s\wedge t)S=sS\cap tS$.
\end{lemma}
\begin{proof}
First, suppose $r'\preceq s,t$. Then $r'S\subset sS\cap tS=rS$, so $r'\preceq r$, and therefore $r$ is the greatest lower bound of $s$ and $t$, i.e. $r=s\wedge t$.

Next, if $s\wedge t$ exists, then by definition $(s\wedge t)S\subset sS\cap tS$. Let $r\in sS\cap tS$. Then $r\preceq s,t$, so $r\preceq s\wedge t$, and $r\in (s\wedge t)S$, so $(s\wedge t)S=sS\cap tS$.
\end{proof}

\begin{lemma}\label{lem:equalunion}
Let $S$ be a semigroup, and let $s_1,\ldots, s_n\in S$. If $\bigcup_{i=1}^n s_iS=rS$ for some $r\in S$, then $rS=s_iS$ for at least one $1\leq i\leq n$.
\end{lemma}
\begin{proof}
$rS\subset\bigcup_{i=1}^n s_nS$ is equivalent to $r\in \bigcup_{i=1}^n s_nS$ which implies that $r\in s_jS$ for some $j$. Then $rS\subset s_jS\subset \bigcup_{i=1}^n s_iS=rS$, so $rS=s_jS$.
\end{proof}

\begin{definition}
We say that a semigroup $S$ satisfies \emph{Clifford's condition}\footnote{This is not the same concept as a Clifford semigroup. Clifford's condition is a term coined by Mark Lawson because it plays an important role in the construction of 0-bisimple inverse semigroups, and Clifford was the first to use this in \cite{clifford53}. See also \cite{lawson97} for more on this.} if for any $s,t\in S$, $sS\cap tS=\varnothing$, or there exists an $r\in S$ such that $sS\cap tS=rS$.
\end{definition}

For instance, all free or free abelian semigroups satisfy Clifford's condition. We will see more examples below.

\begin{definition}\label{def:ucondition}
Following Li \cite{li11}, we say that $J(S)$ is \emph{independent} or that \emph{the constructible right ideals of $S$ are independent} if for any $X_1,\ldots, X_n,Y\in J(S)$, $\bigcup_{i=1}^n X_n=Y$ implies that $X_i=Y$ for at least one $1\leq i\leq n$.
\end{definition}

\begin{proposition}\label{prop:semigrouplattice}
Let $S$ be a left cancellative semigroup. The following two conditions are equivalent.
\begin{enumerate}
\item $S$ satisfies Clifford's condition
\item For every $s,t\in S$ with $t^{-1}(sS)$ nonempty, there is some $r\in S$ such that $t^{-1}(sS)=rS$.
	\newcounter{enumi_saved}
	\setcounter{enumi_saved}{\value{enumi}}
\end{enumerate}
These conditions impliy that $J(S)\cup\{\varnothing\}=\{sS:s\in S\}\cup\{\varnothing\}$ and that $J(S)$ is independent.
If $S$ is algebraically ordered, then (i) is equivalent to the following statement.
\begin{enumerate}
	\setcounter{enumi}{\value{enumi_saved}}
\item Every pair of elements in $S$ that have a common lower bound have a greatest lower bound.
\end{enumerate}
This implies that when $S$ is an algebraically ordered semigroup satisfying Clifford's condition, $(S^0,\wedge)$ is a semilattice and is isomorphic as a semilattice to $J(S)\cup\{\varnothing\}\simeq J(S)^0$.
\end{proposition}
\begin{proof}
(i)$\Rightarrow$(ii): Since $S$ is left cancellative, then for any $X\subset S$ we have $tt^{-1}(X)=tS\cap X$ and $t^{-1}(tX)=X$. If $t^{-1}(sS)$ is nonempty, then so is $tt^{-1}(sS)=sS\cap tS$. Let $q\in S$ be such that $qS=sS\cap tS$. Since $q\in tS$, $t^{-1}(\{q\})$ is nonempty and contains a unique element $r$ since $S$ was left cancellative. Now we have
\begin{align*}
rS=(t^{-1}\{q\})S&=\{uv: u,v\in S, tu=q\}\\
&=t^{-1}\{tuv:u,v\in S, tu=q\}\\
&=t^{-1}\{qv:v\in S\}=t^{-1}(qS)\\
&=t^{-1}(sS\cap tS)=t^{-1}(sS)
\end{align*}

(ii)$\Rightarrow$(i): If $tS\cap sS$ is nonempty, then so is $tt^{-1}(sS)$ and $t^{-1}(sS)$. By assumption, $t^{-1}(sS)=rS$, so $tS\cap sS=(tr)S$.

That (i)+(ii) implies $J(S)\cup\{\varnothing\}=\{sS:s\in S\}\cup\{\varnothing\}$ is a simple induction proof. That this again implies that $J(S)$ is independent follows from Lemma \ref{lem:equalunion}: If $\bigcup_{i=1}^n s_iS=Y\in J(S)$, $Y=rS$ for some $r\in S$, and Lemma \ref{lem:equalunion} gives that $rS=s_iS$ for at least one $i$.

(i)$\Leftrightarrow$(iii): Let $s,t\in S$.  Then $sS\cap tS\neq\varnothing$ if and only if there is some $r\in S$ such that $r\preceq s,t$ if and only if $s$ and $t$ have a common lower bound. By Lemma \ref{lem:semigroupwedge} $s$ and $t$ have a greatest lower bound $s\wedge t$ if and only if $sS\cap tS=(s\wedge t)S$. 

By going to $S^0$, we have $sS^0\cap tS^0=\{0\}=0S^0$ if and only if $sS\cap tS=\varnothing$. Otherwise $sS^0\cap tS^0=rS^0$ for some $r\in S$. The isomorphism from $(S^0,\wedge)$ to $J(S)\cup\{\varnothing\}$ is then constructed by sending $s$ to $sS$ for $s\in S$ and $0$ to $\varnothing$. This is injective since $S$ was algebraically ordered.
\end{proof}

\begin{definition}\label{def:quasilatticeordered}
Let $G$ be a group and $S\subset G$ a subsemigroup. If $S$ is algebraically ordered and generates $G$, it induces a partial order on all of $G$ by $g\leq h$ iff $g^{-1}h\in S$. Nica \cite{nica92} calls $(G,S)$ for \emph{quasilattice ordered} if in addition any finite family of elements in $G$ that have a common upper bound in $S$ has a least common upper bound in $S$. $S$ is called the \emph{positive cone} in $(G,S)$.
\end{definition}

Note that when restricted to $S$, $\leq$ is the same as our $\succeq$. This shows that if $S$ is a positive cone in a quasilattice orderd group, any pair in $S$ that have a common lower bound in $S$ with respect to $\preceq$ have a greatest lower bound in $S$. So $S$ satisfies Clifford's condition by Proposition \ref{prop:semigrouplattice} and it follows that $J(S)$ is independent. Note that Li proves in \cite{li11} that the positive cones of the quasilattice ordered groups have independent constructible right ideals.

\vspace{10pt}
We can give a description of when the $ax+b$ semigroup over an integral domain $R$ satisfies Clifford's condition. The $ax+b$ semigroup over $R$, denoted $R\rtimes R^\times$ is defined to be the set $R\times R^\times$ with product $(b,a)(d,c)=(b+ad,ac)$. Here $R^\times=R\setminus\{0\}$. The reason one considers integral domains is that the $ax+b$ semigroups over these are left cancellative.

Consider first the multiplicative semigroup $(R^\times,\cdot)$. This is a semigroup since $R$ has no zero divisors. We see that for $a,b\in R^\times$, $a\succeq b$ if and only if $a$ divides $b$.

\begin{definition}
A \emph{common multiple} of $a,b\in R^\times$ is an element $c$ of $R^\times$ that is divided by $a$ and $b$. A \emph{least common multiple} of $a$ and $b$ is a common multiple $c$ such that if $c'$ is a common multiple of $a$ and $b$, then $c$ divides $c'$.
\end{definition}

It follows by a similar argument to that in Lemma \ref{lem:semigroupwedge} that $aR^\times\cap bR^\times =cR^\times$ if and only if $c$ is a least common multiple of $a$ and $b$. Note that since $R$ is commutative, $ab\in aR^\times\cap bR^\times\neq\varnothing$. So $R^\times$ satisfies Clifford's condition if and only if every pair in $R^\times$ has a least common multiple (see also Theorem 2.1 in \cite{chapman_glaz00}). Such an integral domain $R$ is often called a GCD domain because one can show that every pair has a greatest common divisor if and only if every pair has a least common multiple. See \cite{chapman_glaz00} for a detailed discussion of GCD domains. They are also discussed in \cite{bourbaki72} where they are called pseudo-Bezout domains. The next lemma is stated without proof in Li's article, but we include it for completeness.

\begin{lemma}
Let $R$ be a ring. For any subrings $I,J\subset R$ and $b,d\in R$, either $(b+I)\cap (d+J)=\varnothing$ or there is some $x\in R$ such that $(b+I)\cap (d+J)=x+I\cap J$.
\end{lemma}
\begin{proof}
Suppose $(b+I)\cap (d+J)\neq\varnothing$. Then there are $y\in I$ and $z\in J$ such that $b+y=d+z$. Write $x=b+y=d+z$. Then $x+I=b+y+I= b+I$ and $x+J=d+z+J=d+J$. So $(b+I)\cap (d+J)=(x+I)\cap (x+J)=x+I\cap J$.
\end{proof}

\begin{proposition}\label{prop:axpbclifford}
Let $R$ be an integral domain. Then $R\rtimes R^\times$ satisfies Clifford's condition if and only if $R$ is a GCD domain.
\end{proposition}
\begin{proof}
We show that $R\rtimes R^\times$ satisfies Clifford's condition if and only if $R^\times$ satisfies Clifford's condition. Suppose $R^\times$ satisfies Clifford's condition. Note that since $R$ is a commutative ring, $aR$ is an ideal of $R$ for every $a\in R$. Let $a,b,c,d\in R$. Then
\begin{align*}
(b,a)R\rtimes R^\times\cap (d,c)R\rtimes R^\times&=[(b+aR)\times aR^\times]\cap[(d+cR)\times cR^\times]\\
&=[(b+aR)\cap (d+cR)]\times [aR^\times \cap cR^\times]
\end{align*}
If this set is nonempty, $(b+aR)\cap (d+cR)$ is nonempty, so by the previous lemma there exists some $x\in R$ satisfying $(b+aR)\cap (d+cR)=x+aR\cap cR$. Moreover, $aR^\times\cap cR^\times\neq \varnothing$, so since $R^\times$ satisfies Clifford's condition, there is some $y\in R^\times$ satisfying $aR^\times\cap cR^\times=yR^\times$. This also implies $aR\cap cR=yR$. So we get
\[
(b,a)R\rtimes R^\times\cap (d,c)R\rtimes R^\times=(x+yR)\times yR^\times=(x,y)R\rtimes R^\times
\]

Suppose $R\rtimes R^\times$ satisfies Clifford's condition and let $a,c\in R^\times$. Since $aR^\times\cap cR^\times\neq\varnothing$,
\[
(0,a)(R\rtimes R^\times)\cap (0,c)(R\rtimes R^\times)= (aR\cap cR)\times (aR^\times\cap cR^\times)\neq \varnothing
\]
It follows that there exist $y\in R^\times$ and $x\in R$ (one may take $x=0$) such that
\[
(0,a)(R\rtimes R^\times)\cap (0,c)(R\rtimes R^\times)=(x,y)R\rtimes R^\times
\]
This implies that $aR^\times\cap cR^\times=yR^\times$.
\end{proof}

Li shows that when $R$ is a Dedekind domain, $J(R\rtimes R^\times)$ is independent. Every Dedekind domain that is also a GCD domain is a principal ideal domain. One way to see this is to use that every nontrivial ideal in a Dedekind domain $R$ is on the form $c^{-1}(aR)$ for some $c,a\in R$. This is for instance proved in \cite{li11}. Note that Li denotes $c^{-1}(aR)$ as $((c^{-1}a)\cdot R)\cap R$. This comes from viewing $c^{-1}a$ as an element of the field of fractions of $R$. Applying statement (ii) in Proposition \ref{prop:semigrouplattice} to the semigroup $R^\times$ one can deduce that if $R$ is also a GCD domain, any nontrivial ideal in $R$ is on the form $aR$ for some $a\in R$. This is the definition of a principal ideal domain.

There exist Dedekind domains that are not principal ideal domains. An example of this is $\mathbb{Z}[\sqrt{10}]$ as seen on p. 407 in \cite{hungerford74}. This shows that not every left cancellative semigroup with independent constructible right ideals satisfies Clifford's condition. On the other hand, every Dedekind domain is Noetherian (Theorem 6.10 in \cite{hungerford74}), but not every GCD domain is Noetherian. So the integral domain $R$ does not have to be a Dedekind domain for $J(R\rtimes R^\times)$ to be independent. Examples of non-Noetherian GCD domains can be found in \cite{chapman_glaz00}.

%% file: inversesemirep.tex
\subsection{The $C^\ast$-algebras of an inverse semigroup}

Let $P$ be an inverse semigroup. We want to recall some common constructions for $C^\ast$-algebras that are generated by representations of $P$ by partial isometries. This is a short account of the theory. A more thorough account can for instance be found in \cite{paterson99} or \cite{duncan_paterson85}. One may construct such $C^\ast$-algebras by associating them to certain groupoids, but we won't use this approach in the present paper.

Let $\{\delta_p\}_{p\in P}$ be the canonical basis of $\ell^2(P)$ satisfying
\[
\iprod{\delta_p,\delta_q}=\begin{cases}1\mbox{ if }p=q\\0\mbox{ otherwise }\end{cases}
\]
Let $\mathbb{C}P$ be the vector space consisting of formal sums
\[
\sum_{i=1}^n a_ip_i
\]
for any $n\in\mathbb{N}$, $a_i\in\mathbb{C}$ and $p_i\in P$. Define an involution on $\mathbb{C}P$ by
\[
\left(\sum_{i=1}^n a_ip_i\right)^\ast=\sum_{i=1}^n \overline{a_i}p_i^\ast
\]
and a product by
\[
\left(\sum_{i=1}^n a_ip_i\right)\left(\sum_{j=1}^m b_jq_j\right)=\sum_{i=1}^n\sum_{j=1}^m a_ib_jp_iq_j
\]
These operations make $\mathbb{C}P$ a $\ast$-algebra. The left regular representation of $\mathbb{C}P$ is defined to be the map $\Lambda:\mathbb{C}P\to B(\ell^2(P))$ given by
\[
\Lambda(p)\delta_q=
\begin{cases}
\delta_{pq}\mbox{ if }p^\ast pq=q\\
0\mbox{ otherwise }
\end{cases}
\]
Then $\Lambda$ can be shown to be a faithful $\ast$-representation of $\mathbb{C}P$. Define $C^\ast_r(P)$ to be the closure of the image of $\Lambda$ with respect to the operator norm.

One way to construct the full $C^\ast$-algebra of $P$ is to show that $\mathbb{C}P$ is dense in the convolution algebra $\ell^1(P)$. One then lets $C^\ast(P)$ be the universal $C^\ast$-enveloping algebra of the Banach $\ast$-algebra $\ell^1(P)$. The left regular representation $\Lambda$ extends to a $\ast$-homomorphism $\Lambda:C^\ast(P)\to C^\ast_r(P)$. 

$C^\ast(P)$ is universal for representations of $P$ by partial isometries. If $A$ is a $C^\ast$-algebra, $\Piso(A)$ is the set of partial isometries in $A$ and $f:P\to \Piso(A)$ is a homomorphism onto a subsemigroup of $\Piso(A)$, then there is a $\ast$-homomorphism $\pi:C^\ast(P)\to A$ such that $\pi(p)=f(p)$ for each $p\in P$. This implies that if $P,Q$ are two inverse semigroups, then every homomorphism $f:P\to Q$ extends to a $\ast$-homomorphism $\pi_f:C^\ast(P)\to C^\ast(Q)$.

Note that if $P$ has a $0$, then $\Lambda(0)\delta_0=\delta_0$ and $\Lambda(0)\delta_p=0$ for $p\neq 0$. So $\Lambda(0)\neq 0$ is a one dimensional projection. This is undesireable in some of our later applications, but it is not too difficult go around the problem. If $0_P$ exists, then $\mathbb{C}0_P$ is an ideal in $\mathbb{C}P$, so it is an ideal in $C^\ast(P)$, and $\Lambda(\mathbb{C}0_P)$ is an ideal in $C^\ast_r(P)$. Let $C^\ast_0(P)=C^\ast(P)/\mathbb{C}0_P$ and $C^\ast_{r,0}(P)=C^\ast_r(P)/\Lambda(\mathbb{C}0_P)$.

Since $\Lambda$ sends $a\in C^\ast(P)$ to $\Lambda(\mathbb{C}0_P)$ if and only if $a\in\mathbb{C}0_P$, $\Lambda$ defines a $\ast$-homomorphism $\Lambda_0:C^\ast_0(P)\to C^\ast_{r,0}(P)$. Moreover, if $P$ and $Q$ are inverse semigroups and $f:P^0\to Q^0$ is a $0$-homomorphism, then $\pi_f$ pushes down to $\pi_{f,0}:C^\ast_0(P)\to C^\ast_0(Q)$. It is important to note that if $P$ is an inverse semigroup without $0$, then $C^\ast_0(P^0)\simeq C^\ast(P)$ and $C^\ast_{r,0}(P^0)\simeq C^\ast_r(P)$.

\begin{definition}
The inverse semigroup $P$ is said to have \emph{weak containment} if $\Lambda:C^\ast(P)\to C^\ast_r(P)$ is an isomorphism. Clearly $\Lambda$ is an isomorphism if and only if $\Lambda_0$ is an isomorphism. See \cite{milan10} for a recent study of weak containment for inverse semigroups.
\end{definition}

\begin{proposition}
Let $P$ be a commutative inverse semigroup. Then $P$ has weak containment.
\end{proposition}
\begin{proof}
This does for instance follow from Paterson's results in \cite{paterson78} since every commutative inverse semigroup $P$ is a so-called Clifford semigroup and any subgroup of $P$ has to be amenable.
\end{proof}

\begin{corollary}\label{cor:maxembeddednorm}
Let $P$ be an inverse semigroup, and let $L$ be the subsemilattice of idempotents in $P$. Let $D$ be the $C^\ast$-subalgebra of $C^\ast_r(P)$ generated by $\Lambda(L)$. Then $D$ is canonically isomorphic to $C^\ast_r(L)$ and $C^\ast(L)$. A similar result holds if we look at the subalgebra generated by $\Lambda_0(L)$ in $C^\ast_{r,0}(P)$.
\end{corollary}
\begin{proof}
Since $L$ is commutative it has weak containment, so $C^\ast_r(L)\simeq C^\ast(L)$ is universal for representations of $L$. Thus the norm $\mathbb{C}L$ gets from its representation on $\ell^2(L)$ is greater than or equal to the one it gets from $\ell^2(P)$. But we also have that for $a\in \mathbb{C}L$
\begin{align*}
\norm{\Lambda(a)}_{C^\ast_r(P)}&=\sup\{\norm{\Lambda(a)\delta_p}:p\in P\}\\
&\geq \sup\{\norm{\Lambda(a)\delta_p}:p\in L\}\\
&=\norm{\Lambda(a)}_{C^\ast_r(L)}
\end{align*}
\end{proof}

Let $L$ be a semilattice with $0$ and let $S$ be a set such that there is an injective $0$-homomorphism $f:L\to 2^S$. Here $2^S=\{X:X\subset S\}$ is given the structure of a semilattice by saying that the semigroup product is given by set intersection. This gives a representation $\mu:L\to\ell^\infty(S)$ by $\mu(a)=\chi_{f(a)}$ where $\chi_X$ is the characteristic function of $X\subset S$. Let $C^\ast(L;f)$ be the $C^\ast$-algebra generated by the image of $\mu$. By the universality of $C^\ast_0(L)$ for $0$-representations of $L$ by commuting projections, there is a $\ast$-homomorphism $\pi:C^\ast_0(L)\to C^\ast(L;f)$ such that $\pi(a)=\mu(a)$ for each $a\in L$. We will say that $f$ is a \emph{maximal} representation of $L$ if this $\pi$ is an isomorphism. The fact that $\pi$ isn't always an isomorphism is related to the problem of finding the right way to map semilattices into Boolean algebras which is discussed by Exel in \cite{exel08,exel09}.

Before we investigate this we want to discuss filters, which is a concept that is important in the representation theory of semilattices in the same way that characters are important in the representation theory of abelian $C^\ast$-algebras.

\begin{definition}
Let $L$ be a semilattice with $0$. A \emph{filter} on $L$ is a $0$-homomorphism $\phi:L\to\{0,1\}$. Here $\{0,1\}$ is given the structure of a semilattice with $1\cdot 0=0$. An alternative view of filters on $L$ is to define them to be subsets $\phi\subset L$ such that for all $a,b\in L$,
\begin{enumerate}
\item If $a\in \phi$, then $a\preceq b$ implies $b\in\phi$.
\item If $a,b\in\phi$, then $ab\in\phi$.
\item $0\notin\phi$ and $1\in\phi$.
\end{enumerate}
Through the correspondence $a\in\phi\Leftrightarrow\phi(a)=1$, one sees that these are equivalent definitions. The latter picture is the more traditional one.
\end{definition}

Now if $\psi$ is a character on $C^\ast(L;f)$ it defines a filter $\phi$ on $L$ by saying that $\phi(a)=\psi(\mu(a))$ for each $a\in L$. Since $\mu(a)$ is always an idempotent, we have $\psi(\mu(a))\in\{0,1\}$, so this $\phi$ well defined. Moreover, since $\mu(L)$ generates $C^\ast(L;f)$, two characters on $C^\ast(L;f)$ are equal if and only if their associated filters are equal. So the characters on $C^\ast(L;f)$ are completely determined by their associated filters. In general, not every filter on $L$ will extend to a character on $C^\ast(L;f)$.

\begin{proposition}\label{prop:ucondition}
Let the setup be as above. The following conditions are equivalent
\begin{enumerate}
\item $f:L\to 2^S$ is a maximal representation, i.e. $\pi:C^\ast_0(L)\to C^\ast(L;f)$ is an isomorphism.
\item For every filter $\phi:L\to\{0,1\}$, there is a character $\psi$ on $C^\ast(L;f)$ such that $\psi(\mu(a))=\phi(a)$ for each $a\in L$.
\item For all $a_1,\ldots, a_n\in L$, if there is some $b\in L$ such that $\bigcup_{i=1}^n f(a_i)=f(b)$, then $a_i=b$ for at least one $1\leq i\leq n$.
\end{enumerate}
\end{proposition}
\begin{proof}
(i)$\Rightarrow$(ii): Let $\phi:L\to\{0,1\}$ be a filter. If $C^\ast(L;f)$ is isomorphic to $C^\ast_0(L)$, then by the universal properties of this $C^\ast$-algebra there is a nonzero $\ast$-homomorphism $C^\ast(L;f)\to C^\ast_0(\{0,1\})\simeq\mathbb{C}$ extending $\phi$. A nonzero $\ast$-homomorphism to $\mathbb{C}$ is exactly the definition of a character.

(ii)$\Rightarrow$(iii): Let $a_1,\ldots, a_n,b\in L$ be such that $\bigcup_{i=1}^n f(a_i)=f(b)$. Note that $\mu(b)\leq\sum_{i=1}^n\mu(a_i)$. Let $\phi=\{c\in L:b\preceq c\}$. This is a filter on $L$. Let $\psi$ be the extending character. Then
\[
1=\psi(\mu(b))\leq \sum_{i=1}^n \psi(\mu(a_i))
\]
Then $\psi(\mu(a_i))=\phi(a_i)=1$ for at least one $i$ . We have $f(a_ib)=f(a_i)\cap f(b)=f(a_i)$, so $a_ib=a_i$, that is $a_i\preceq b$. If $\phi(a_i)=1$, then $b\preceq a_i$ by the definition of $\phi$, so $a_i=b$.

(iii)$\Rightarrow$(i): This can be proved just like (i)$\Rightarrow$(ii) in Proposition 2.24 of \cite{li11}, so we skip the proof.
\end{proof}

\begin{corollary}\label{cor:ucondition}
Let $S$ be a left cancellative semigroup. Then $J(S)$ is independent (Definition \ref{def:ucondition}) if and only if the inclusion $\iota:J(S)\cup\{\varnothing\}\to 2^S$ is a maximal representation.
\end{corollary}
\begin{proof}
This follows from (iii) above and the definition of the independence of $J(S)$.
\end{proof}

We will need the following proposition later.

\begin{proposition}\label{prop:inversesemiconditional}
Let $P$ be an $E^\ast$-unitary inverse semigroup, and let $L$ be its subsemilattice of idempotents. There exists a faithful conditional expectation $E_{r,0}:C^\ast_{r,0}(P)\to C^\ast_0(L)$ such that $E_{r,0}(\Lambda_0(p))=p$ if $p\in L$ and $E_{r,0}(\Lambda_0(p))=0$ otherwise.
\end{proposition}
\begin{proof}
Let $E_r:B(\ell^2(P))\to\ell^\infty(P)$ be the usual faithful conditional expectation given by
\[
\iprod{E_r(a)\delta_q,\delta_q}=\iprod{a\delta_q,\delta_q}
\]
Here $\ell^\infty(P)$ is viewed as a subalgebra of $B(\ell^2(P))$ represented by pointwise multiplication. First, if $p\in L$, then $\iprod{\Lambda(p)\delta_q,\delta_r}\neq 0$ if and only if $p^\ast pq=pq=q$ and $pq=r$ which implies $r=q$. So $\Lambda(p)\in\ell^\infty(P)$, and $E_r(\Lambda(p))=\Lambda(p)$.

In general, let $p\in P$ and suppose $\iprod{\Lambda(p)\delta_q,\delta_q}\neq 0$ for some $q\in P\setminus\{0\}$. Then $pq=q$, so since $P$ is $E^\ast$-unitary, $p\in L$.

Due to Corollary \ref{cor:maxembeddednorm}, we now identify $C^\ast(L)$ with the closure of $\mathbb{C}L$ inside $C^\ast_r(P)$. We have $E_r:C^\ast_r(P)\to C^\ast(L)$, and since $E_r(\Lambda(0))=\Lambda(0)$, $E_{r,0}:C^\ast_{r,0}(P)\to C^\ast_0(L)$ can be defined with the desired properties.

For any $a\in C^\ast_r(P)$, let $[a]$ denote its image in $C^\ast_{r,0}(P)$. We have that $E_{r,0}([a^\ast a])=0$ if and only if $E_r(a^\ast a)=\alpha \Lambda(0_P)$ for some $\alpha\in\mathbb{C}$ if and only if $\iprod{a^\ast a\delta_{0_P},\delta_{0_P}}=\norm{a\delta_{0_P}}^2=\alpha$ and $\norm{a\delta_q}^2=0$ for all $q\in P\setminus\{0_P\}$. This implies that $a^\ast a=\alpha\Lambda(0_P)$ and that $[a^\ast a]=0$, so $E_{r,0}$ is faithful.
\end{proof}

On the other hand we have the following lemma, which is also interesting.

\begin{lemma}\label{lem:conditionalexpectationestar}
Let $A$ be a $C^\ast$-algebra generated by an inverse semigroup $P$ of partial isometries. Let $L$ be the semilattice of idempotents in $P$, and let $D$ be the subalgebra of $A$ generated by $L$. If there exists a conditional expectation $E:A\to D$ such that $E(p)=p$ if $p\in L$ and $E(p)=0$ otherwise, then $P$ is $E^\ast$-unitary.
\end{lemma}
\begin{proof}
Suppose $p\in P$ and $q\in L\setminus\{0\}$ satisfy $pq=q$. Then $pq=q=E(q)=E(pq)=E(p)q$ since $E$ is a conditional expectation. This implies that $E(p)\neq 0$, so $p\in L$. This shows that $P$ is $E^\ast$-unitary.
\end{proof}

%% file: semigrouprep.tex
\subsection{The left regular representation and the left inverse hull of a left cancellative semigroup}

From now on, $S$ will always be a left cancellative semigroup unless something else is stated. Let $\{\varepsilon_s\}_{s\in S}$ be the orthogonal basis of $\ell^2(S)$, where $\varepsilon_s(t)=1$ if $s=t$, and $0$ otherwise. The left regular representation of $S$ is the semigroup homomorphism $s\mapsto V_s$, where $V_s:\ell^2(S)\to\ell^2(S)$ is given by $V_s\varepsilon_t=\varepsilon_{st}$.

Now, $\iprod{V_s^\ast\varepsilon_t,\varepsilon_r}=1$ if $t=sr$ and $0$ otherwise, so
\[
V_s^\ast\varepsilon_t=\sum_{r\in t^{-1}(\{s\})}\varepsilon_r
\]
Since $S$ is left cancellative, $t^{-1}(\{s\})$ is either a singleton or empty. It follows readily that $V_s$ is an isometry for each $s\in S$.

Let $E:B(\ell^2(S))\to\ell^\infty(S)$ be the conditional expectation given by $\iprod{E(a)\varepsilon_s,\varepsilon_s}=\iprod{a\varepsilon_s,\varepsilon_s}$ for each $s\in S$. Here we view $\ell^\infty(S)$ as a subalgebra of $B(\ell^2(S))$ represented by pointwise multiplication. For a subset $X\subset S$, let $\chi_X\in\ell^\infty(S)$ be the associated characteristic function. It is easy to check that for all $s\in S$ and $X\subset S$
\begin{equation}\label{eq:covariance}
V_s\chi_X V_s^\ast=\chi_{sX},\qquad V_s^\ast\chi_X V_s=\chi_{s^{-1}(X)}
\end{equation}
In particular, $V_sV_s^\ast=\chi_{sS}$.

We will let $C^\ast_r(S)$ be the $C^\ast$-algebra generated by $\{V_s:s\in S\}$, and let $D_r(S)$ be the commutative $C^\ast$-algebra generated by $\{\chi_X:X\in J(S)\}$. Note that $C^\ast_r(S)$ is the closed linear span of the set
\[
V(S)=\{V_{t_1}^\ast V_{s_1}\cdots V_{t_n}^\ast V_{s_n}:n\in\mathbb{N},s_1,\ldots, s_n,t_1,\ldots, t_n\in S\} 
\]
$V(S)$ is itself a semigroup under composition of operators, and it is an inverse semigroup since it consists of partial isometries with commuting initial and final projections. To see this, note that if $V=V_{t_1}^\ast V_{s_1}\cdots V_{t_n}^\ast V_{s_n}$, then by repeatedly applying the relations in \eqref{eq:covariance} we get that $VV^\ast=\chi_X$ with $X=t_1^{-1}s_1\cdots t_n^{-1}s_nS$. Similarly, $V^\ast V=\chi_Y$ with $Y=s_n^{-1}t_n\cdots s_1^{-1}t_1S$. The second assertion of the next lemma is also proved in \cite{li11}.

\begin{lemma}\label{lem:simplifiedjs}
The semilattice of idempotents in $V(S)$ is isomorphic to the $\cap$-semilattice
\[
J:=\{t_1^{-1}s_1\cdots t_n^{-1}s_nS:n\in\mathbb{N}, s_i,t_i\in S\}
\]
Moreover $J=J(S)$.
\end{lemma}
\begin{proof}
We saw in the previous paragraph that the semilattice of idempotents in $V(S)$ is $\{\chi_X:X\in J\}$. For any $X,Y\in J$, $\chi_X\chi_Y=\chi_{X\cap Y}$, so $J$ has to be a $\cap$-semilattice and be isomorphic to $\{\chi_X:X\in J\}$. Since $J$ is therefore closed under $\cap$ it must be equal to $J(S)$ as defined in subsection \ref{sec:semilattice}.
\end{proof}

The inverse semigroup $V(S)$ will play an important role in the sequel. As we noted in the introduction it can be given a purely algebraic description. Let $I(S)$ be the inverse semigroup of all bijective partially defined functions $S\to S$. Define $I_l(S)$ to be the inverse subsemigroup of $I(S)$ generated by the partial bijections $\{\lambda_s\}_{s\in S}$ where $\lambda_s:S\to sS$ is given by $\lambda_s(t)=st$.

\begin{lemma}\label{lem:isomorphicinverse}
There is a faithful representation $\omega:I_l(S)\to B(\ell^2(S))$ such that $\omega(\lambda_s)=V_s$ for each $s\in S$. So $\omega$ is an isomorphism of $I_l(S)$ onto $V(S)$.
\end{lemma}
\begin{proof}
For any $f\in I_l(S)$, define $\omega(f):\ell^2(S)\to\ell^2(S)$ by
\[
\omega(f)\varepsilon_s=\begin{cases}\varepsilon_{f(s)}\mbox{ if }s\in\dom(f)\\0\mbox{ otherwise }\end{cases}
\]
where $\dom(f)$ is the domain of $f$. Then $\omega(f)\in B(\ell^2(S))$ since $f$ is injective, and for any $s\in S$, $\omega(\lambda_s)=V_s$. For any $f,g\in I_l(S)$ we have that $f=g$ if and only if $\dom(f)=\dom(g)$ and $f(s)=g(s)$ for all $s\in\dom(f)$. This happens if and only if $\ker\omega(f)=\ker\omega(g)$ and $\omega(f)\varepsilon_s=\omega(g)\varepsilon_s$ for all $s\in S$. This is equivalent to $\omega(f)=\omega(g)$. So $\omega$ is an injective map. Now for any $f,g\in I_l(S)$ and $s\in S$,
\[
\omega(f)\omega(g)\varepsilon_s=\begin{cases}\varepsilon_{f(g(s))}\mbox{ if }s\in g^{-1}(\dom(f))\\0\mbox{ otherwise }\end{cases}
\]
so $\omega(f\cdot g)=\omega(f)\omega(g)$. This shows that $\omega$ is a surjective homomorphism of $I_l(S)$ onto $V(S)$, and thus it is an isomorphism.
\end{proof}

$I_l(S)$ is often called the \emph{left inverse hull} of $S$ and has previously been studied in several settings. Some recent information on it can be found in \cite{meakin11, lawson97, jiang03} and \cite{lawson99}. Using $\omega$ one can translate most statements about $V(S)$ into statements about $I_l(S)$ and vice versa. Note that the semilattice of idempotents in $I_l(S)$ is $\{i_X:X\in J(S)\}\simeq J(S)$ where $i_X:X\to X$ is the identity map on $X$. We will sometimes identify $J(S)$ with $\{i_X:X\in J(S)\}$ from here on.

Many of the ideas of this subsection are present in \cite{li11}, but they are not expressed in terms of $V(S)$ as an inverse semigroup. Using a simple induction argument (or deducing it from the proof of Lemma \ref{lem:isomorphicinverse}) we know that for any $V\in V(S)$ and $s\in S$, $V\varepsilon_s$ is either $0$ or $\varepsilon_t$ for some $t\in S$.

\begin{lemma}\label{lem:eofvisv}
Let $V\in V(S)$. Then $E(V)=V$ if and only if $V$ is idempotent.
\end{lemma}
\begin{proof}
By Lemma \ref{lem:simplifiedjs}, $V$ is idempotent if and only if $V=E_X$ for some $X\in J(S)$. This implies $E(V)=V$. Let $V\in V(S)$, and suppose $E(V)=V$. For every $s\in S$, $\iprod{V\varepsilon_s,\varepsilon_s}$ is either $0$ or $1$, so $V=E(V)=\chi_X$ where $X=\{s\in S:\iprod{V\varepsilon_s,\varepsilon_s}=1\}$. Hence $V=V^2$.
\end{proof}

\begin{corollary}
Let $a\in C^\ast_r(S)$. Then $E(a)=a$ if and only if $a\in D_r(S)$.
\end{corollary}
\begin{proof}
$V(S)$ spans a dense subset of $C^\ast_r(S)$, and $\{E_X:X\in J(S)\}$ spans a dense subset of $D_r(S)$. The result follows by the linearity and continuity of $E$.
\end{proof}

\begin{lemma}\label{lem:rightactioncommutes}
Let $\rho$ be the right action of $S$ on itself given by $\rho_r(s)=sr$ for $s,r\in S$. Then for every $f\in I_l(S)$ and $s\in\dom(f)$,
\begin{equation}\label{eq:rightactioncommutes}
\rho_r(f(s))=f(\rho_r(s))
\end{equation}
for all $r\in S$.
\end{lemma}
\begin{proof}
Note that since $\dom(f)=\dom(f^\ast f)\in J(S)$ is a right ideal in $S$, $s\in\dom(f)$ implies $sr\in\dom(f)$ for all $r\in S$. Equation \eqref{eq:rightactioncommutes} clearly holds when $f=\lambda_t$ for some $t\in S$. Suppose now that $f=\lambda_t^\ast$. Let $s\in\dom(\lambda_t^\ast)=tS$. Then there is some $q\in S$ such that $tq=s$ and $\lambda_t^\ast(s)=q$. For any $p\in S$, $\lambda_t^\ast(\rho_r(s))=p$ if and only if $sr=\rho_r(s)=tp$. On the other hand, $\rho_r(\lambda_t^\ast(s))=qr=p$ if and only if $tqr=tp$. Since $s=tq$, this happens if and only if $sr=tp$. So for any $p\in S$, $\rho_r(\lambda_t^\ast(s))=p$ if and only if $\lambda_t^\ast(\rho_r(s))=p$. This shows that $\rho_r(\lambda_t^\ast(s))=\lambda_t^\ast(\rho_r(s))$.

Let $f\in I_l(S)$ be arbitrary. Then there are $n\in\mathbb{N}$ and $s_1,\ldots,s_n,t_1,\ldots,t_n\in S$ such that $f=\lambda_{t_1}^\ast\lambda_{s_1}\cdots\lambda_{t_n}^\ast\lambda_{s_n}$. Since $s\in\dom(f)$,
\[
\lambda_{t_j}^\ast\lambda_{s_j}\cdots\lambda_{t_n}^\ast\lambda_{s_n}(s)\in\dom(\lambda_{t_1}^\ast\lambda_{s_1}\cdots\lambda_{t_{j-1}}^\ast\lambda_{s_{j-1}})
\]
for all $1\leq j\leq n$. So
\[
\rho_r(\lambda_{t_j}^\ast\lambda_{s_j}\cdots\lambda_{t_n}^\ast\lambda_{s_n}(s))\in\dom(\lambda_{t_1}^\ast\lambda_{s_1}\cdots\lambda_{t_{j-1}}^\ast\lambda_{s_{j-1}})
\]
for all $1\leq j\leq n$ since $\dom(\lambda_{t_1}^\ast\lambda_{s_1}\cdots\lambda_{t_{j-1}}^\ast\lambda_{s_{j-1}})$ is a right ideal in $S$. Starting with
\[
\lambda_{t_1}^\ast\lambda_{s_1}\cdots\lambda_{t_n}^\ast\lambda_{s_n}(\rho_r(s))
\]
we can then move $\rho_r$ to the left one step at a time until we get
\[
\rho_r(\lambda_{t_1}^\ast\lambda_{s_1}\cdots\lambda_{t_n}^\ast\lambda_{s_n}(s))
\]
\end{proof}

\begin{corollary}\label{cor:liftrelation}
Let $f\in I_l(S)$. For any $s\in\dom(f)$, $f\lambda_s=\lambda_{f(s)}$
\end{corollary}
\begin{proof}
Note first that $\dom f\lambda_s=\dom\lambda_{f(s)}=S$. By Lemma \ref{lem:rightactioncommutes} we get that for any $r\in S$,
\[
f\lambda_s(r)=f(sr)=f(\rho_r(s))=\rho_r(f(s))=f(s)r=\lambda_{f(s)}(r)
\]
So $f\lambda_s(r)=\lambda_{f(s)}(r)$ for all $r\in S$, hence $f\lambda_s=\lambda_{f(s)}$.
\end{proof}

\begin{lemma}\label{lem:reducedcondexp}
$V(S)$ is $E^\ast$-unitary if and only if for every $V\in V(S)$ we have $E(V)=V$ or $E(V)=0$.
\end{lemma}
\begin{proof}
Assume $V(S)$ is $E^\ast$-unitary and let $V\in V(S)$. We show that if $V\varepsilon_r=\varepsilon_r$ for some $r\in S$, then $E(V)=V$. Otherwise $E(V)$ is of course $0$. Suppose $V\varepsilon_r=\varepsilon_r$. Using $\omega$ we can get from Corollary \ref{cor:liftrelation} that $VV_r=V_r$. Since $V(S)$ is $E^\ast$-unitary this implies that $V$ is idempotent, so $E(V)=V$ by Lemma \ref{lem:eofvisv}. The converse statement follows from Lemma \ref{lem:conditionalexpectationestar}.
\end{proof}

Note that this implies that if $V(S)$ is $E^\ast$-unitary, then $E(C^\ast_r(S))=D_r(S)$. Of course, $V(S)$ is $E^\ast$-unitary if and only if $I_l(S)$ is $E^\ast$-unitary.

\begin{lemma}\label{lem:sequalst}
Let $s,t\in S$. The following conditions are equivalent.
\begin{enumerate}
\item $s=t$
\item $\lambda_s\lambda_t^\ast$ is idempotent.
\item $\lambda_t^\ast \lambda_s=1$.
\end{enumerate}
\end{lemma}
\begin{proof}
(i)$\Rightarrow$(ii) is trivial.

(ii)$\Rightarrow$(iii): We have $\lambda_s\lambda_t^\ast \lambda_s\lambda_t^\ast=\lambda_s\lambda_t^\ast$. Left multiplying with $\lambda_s^\ast$ and right multiplying with $\lambda_t$ gives the desired equality.

(iii)$\Rightarrow$(i): This implies that $\lambda_s\lambda_t^\ast \lambda_s=\lambda_s$ and $\lambda_t^\ast \lambda_s \lambda_t^\ast=\lambda_t^\ast$. Since these relations are unique for $\lambda_s^\ast$ we get that $\lambda_t^\ast=\lambda_s^\ast$. So $\lambda_s=\lambda_t$ and $s=\lambda_s(1)=\lambda_t(1)=t$.
\end{proof}

\begin{corollary}\label{cor:eastunitarycancellative}
If $I_l(S)$ is $E^\ast$-unitary, $S$ is cancellative.
\end{corollary}
\begin{proof}
Let $s,t,r,p\in S$ and assume that $sr=tr=p$. Since $\lambda_s\lambda_r=\lambda_p$,  $\lambda_s^\ast \lambda_p=\lambda_r$. So $\lambda_t\lambda_s^\ast \lambda_p=\lambda_t\lambda_r=\lambda_p$. Since $I_l(S)$ is $E^\ast$-unitary $\lambda_t\lambda_s^\ast$ is idempotent, so by the previous lemma $s=t$. Hence $S$ is (right) cancellative.
\end{proof}

\begin{lemma}\label{lem:subsemigroupvofs}
Assume $S$ is a subsemigroup of a group $G$. Let $m,n\in\mathbb{N}$ and $s_i,t_i,p_j,q_j\in S$ for $1\leq i\leq n$, $1\leq j\leq m$. Set $f=\lambda_{t_1}^\ast \lambda_{s_1}\cdots \lambda_{t_n}^\ast \lambda_{s_n}$ and $f'=\lambda_{q_1}^\ast \lambda_{p_1}\cdots \lambda_{q_m}^\ast \lambda_{p_m}$, and assume $f,f'\neq 0$. If $f=f'$ then the equality
\begin{equation}\label{eq:groupequality}
t_1^{-1}s_1\cdots t_n^{-1}s_n=q_1^{-1}p_1\cdots q_m^{-1}p_m
\end{equation}
holds in $G$, where $(\cdot)^{-1}$ means taking inverses in $G$.
\end{lemma}
\begin{proof}
Since $f,f'\neq 0$ we can pick some $r\in \dom(f)$. Then the equality $f(r)=f'(r)$ gives
\[
t_1^{-1}s_1\cdots t_n^{-1}s_nr=q_1^{-1}p_1\cdots q_m^{-1}p_mr
\]
where $(\cdot)^{-1}$ denotes the preimage by left multiplication in $S$. However this implies that the same relation holds in $G$, where $(\cdot)^{-1}$ now stands for the inverse operation in $G$. Cancelling with $r$ we get \eqref{eq:groupequality}. 
\end{proof}

The proof of the next proposition uses techniques similar to those employed by Jiang in \cite{jiang03}.

\begin{proposition}\label{prop:groupstrongeastunitary}
Let $S$ be a left cancellative semigroup. Then $S$ embeds into a group if and only if $I_l(S)$ is strongly $E^\ast$-unitary.
\end{proposition}
\begin{proof}
Suppose first that $S$ embeds into a group $G$. We omit writing the embedding homomorphism, and instead view $S$ as a subsemigroup of $G$. Define a grading $\varphi:I_l(S)^0\to G^0$ by
\begin{align*}
\varphi(0)&=0\\
\varphi(\lambda_{t_1}^\ast \lambda_{s_1}\cdots \lambda_{t_n}^\ast \lambda_{s_n})&=t_1^{-1}s_1\cdots t_n^{-1}s_n\\
&\mbox{ when } \lambda_{t_1}^\ast \lambda_{s_1}\cdots \lambda_{t_n}^\ast \lambda_{s_n} \neq 0
\end{align*}
This is well-defined because of Lemma $\ref{lem:subsemigroupvofs}$. Suppose $\varphi(f)=1$ for some $f\in I_l(S)$. Then if $f=\lambda_{t^1}^\ast \lambda_{s_1}\cdots \lambda_{t_n}^\ast \lambda_{s_n}$, $t_1^{-1}s_1\cdots t_n^{-1}s_n=1$. So $f(r)=t_1^{-1}s_1\cdots t_n^{-1}s_nr=r$ for all $r\in\dom(f)$. Hence $f$ is idempotent, and $\varphi$ is idempotent pure. This shows that $I_l(S)$ is strongly $E^\ast$-unitary.

Now suppose $I_l(S)$ is strongly $E^\ast$-unitary by some idempotent pure grading $\varphi:I_l(S)^0\to G^0$. For any $t\in S$, $1=\varphi(\lambda_t^\ast \lambda_t)=\varphi(\lambda_t^\ast)\varphi(\lambda_t)$, so $\varphi(\lambda_t)^{-1}=\varphi(\lambda_t^\ast)$. For any $s,t\in S$, if $\varphi(\lambda_s)=\varphi(\lambda_t)$, then $\varphi(\lambda_s\lambda_t^\ast)=1$, so $\lambda_s\lambda_t^\ast$ is idempotent, and by Lemma \ref{lem:sequalst} $s=t$. This implies that the homomorphism $S\to G$ given by $s\mapsto\varphi(\lambda_s)$ is injective.
\end{proof}

We want to find a relation between $C^\ast_{r,0}(I_l(S))$ and $C^\ast_r(S)$.

\begin{lemma}
Let $T:\ell^2(S)\to\ell^2(I_l(S))$ be the isometry defined by
\[
T\varepsilon_s=\delta_{\lambda_s}\qquad s\in S
\]
Then $T^\ast\Lambda(f) T=\omega(f)$ for all $f\in I_l(S)$.
\end{lemma}
\begin{proof}
Let $f\in I_l(S)$ and $s\in\dom(f)$. Then $s\in\dom(f^\ast f)$, so $f^\ast f(s)=s$. By Corollary \ref{cor:liftrelation}, $f^\ast f\lambda_s=\lambda_s$ and $f\lambda_s=\lambda_{f(s)}$. Now by the definition of $\Lambda$,
\[
T^\ast \Lambda(f)T\varepsilon_s=T^\ast\Lambda(f)\delta_{\lambda_s}=T^\ast \delta_{f\lambda_s}=T^\ast\delta_{\lambda_{f(s)}}=\varepsilon_{f(s)}=\omega(f)\varepsilon_s
\]
On the other hand, if $s\notin\dom(f)$ then $s\notin\dom(f^\ast f)$. Thus $f^\ast f \lambda_s\neq \lambda_s$ and we get $T^\ast\Lambda(f)T\varepsilon_s=T^\ast\Lambda(f)\delta_{\lambda_s}=0$. So $T^\ast\Lambda(f)T\varepsilon_s=\omega(f)\varepsilon_s$ for any $s\in S$. This shows that $T^\ast\Lambda(f)T=\omega(f)$ for any $f\in I_l(S)$.
\end{proof}

\begin{corollary}
There is a surjective $\ast$-homomorphism $h:C^\ast_{r,0}(I_l(S))\to C^\ast_r(S)$ such that $h(\Lambda_0(f))=\omega(f)$ for all $f\in I_l(S)$.
\end{corollary}
\begin{proof}
Define $h':C^\ast_r(I_l(S))\to C^\ast_r(S)$ by $h'(a)=T^\ast a T$. Then $h'$ is a $\ast$-homomorphism on the span of $\Lambda(I_l(S))$. Since this span is dense in $C^\ast_r(I_l(S))$ and since $h'$ is continuous, it has to be a $\ast$-homomorphism on all of $C^\ast_r(I_l(S))$. Since $h'$ sends $\Lambda(0)$ to $0$ whenever $0\in I_l(S)$, it descends to a $\ast$-homomorphism $h:C^\ast_{r,0}(I_l(S))\to C^\ast_r(S)$ with the desired properties.
\end{proof}

\begin{theorem}\label{thm:ucondisomcstar}
Suppose $I_l(S)$ is $E^\ast$-unitary. Then the map
\[
h:C^\ast_{r,0}(I_l(S))\to C^\ast_r(S)
\]
is an isomorphism if and only if $J(S)$ is independent.
\end{theorem}
\begin{proof}
Recall that $D_r(S)$ is the diagonal subalgebra of $C^\ast_r(S)$ generated by $J(S)$. Then $D_r(S)$ is $C^\ast(J(S)^0;\iota)$ as described in Proposition \ref{prop:ucondition} and the paragraphs before it. Here $\iota:J(S)^0\to 2^S$ is the inclusion map.

The restriction of $h$ to $C^\ast_0(J(S))$ (which we can identify with the subalgebra generated by the image of $J(S)$ in $C^\ast_{r,0}(I_l(S))$ by Corollary \ref{cor:maxembeddednorm}) maps onto $D_r(S)$, and this restriction must necessarily be equal to the map $\pi$ as described in and before Proposition \ref{prop:ucondition}. According to this proposition and Corollary \ref{cor:ucondition}, $h|_{C^\ast_0(J(S))}$ is an isomorphism if and only if $J(S)$ is independent. This means that if $J(S)$ is not independent, $h$ is not injective.

Suppose $J(S)$ is independent and $I_l(S)$ is $E^\ast$-unitary. By Proposition \ref{prop:inversesemiconditional} there is a faithful conditional expectation $E_{r,0}:C^\ast_{r,0}(I_l(S))\to C^\ast_0(J(S))$. As a consequence of Lemma \ref{lem:reducedcondexp}, $E(C^\ast_r(S))=D_r(S)$. Moreover for any $V\in V(S)$, $E(V)=V$ if and only if $V$ is idempotent, and $E(V)=0$ otherwise. Using the properties of $E_{r,0}$ given in Proposition \ref{prop:inversesemiconditional} it follows that $E\circ h=h\circ E_{r,0}$

Now assume that $h(a)=0$ for some $a\in C^\ast_{r,0}(I_l(S))$. Then $h(a^\ast a)=0$, so $E(h(a^\ast a))=0=h(E_{r,0}(a^\ast a))$. Since $h$ is an isomorphism on the image of $E_{r,0}$, $E_{r,0}(a^\ast a)=0$, so $a^\ast a=a=0$ since $E_{r,0}$ is faithful. This shows that $h$ is an isomorphism.
\end{proof}

We can use the equality $D_r(S)=C^\ast(J(S)^0;\iota)$ to describe the characters on $D_r(S)$. Proposition \ref{prop:ucondition} and Corollary \ref{cor:ucondition} imply that when $J(S)$ is independent, the characters on $C^\ast(J(S)^0;\iota)$ are uniquely determined by the filters on $J(S)^0$. When $S$ is algebraically ordered and satisfies Clifford's condition, Proposition \ref{prop:semigrouplattice} tells us that $(S^0,\wedge)\simeq J(S)^0$. So in this case the characters on $D_r(S)$ correspond to the filters on $(S^0,\wedge)$. It is not difficult to see that the filters on $(S^0,\wedge)$ are exactly what Nica calls non-void hereditary directed subsets of $S$ in subsection 6.2 of \cite{nica92}. This is sometimes called the Nica spectrum of $S$. In general, the set of characters on $D_r(S)$ corresponds to some subset of the set of filters on $J(S)^0$, but it is not always obvious what this subset is.

Doing computations in $I_l(S)$ can be difficult, but if $S$ satisfies Clifford's condition it becomes easier. Note that $S$ satisfies Clifford's condition exactly when $I_l(S)$ is $0$-bisimple. We will however not use this fact explicitly in this paper. See for instance \cite{clifford53,lawson97} or \cite{jiang03} for more information on this. 

\begin{proposition}\label{prop:cliffordsimplification}
The following conditions are equivalent
\begin{enumerate}
\item $S$ satisfies Clifford's condition.
\item For all $s,t\in S$ such that $\lambda_t^\ast \lambda_s\neq 0$ there exist $p,q\in S$ such that $\lambda_t^\ast \lambda_s=\lambda_p\lambda_q^\ast$.
\item $I_l(S)\setminus\{0\}=\{\lambda_p \lambda_q^\ast:p,q\in S\}$.
\end{enumerate}
\end{proposition}
\begin{proof}
(i)$\Rightarrow$(ii): Let $s,t\in S$. If $\lambda_t^\ast \lambda_s\neq 0$, then $sS\cap tS\neq \varnothing$, so $sS\cap tS=rS$ for some $r\in S$. Since $r\in sS\cap tS$, $p:=t^{-1}(r)$ and $q:=s^{-1}(r)$ exist, and $\lambda_t\lambda_p=\lambda_r$, so we get that $\lambda_t^\ast \lambda_r=\lambda_t^\ast\lambda_t\lambda_p=\lambda_p$. Similarly $\lambda_s^\ast \lambda_r=\lambda_q$. By the definition of $r$ we have
\[
\lambda_t\lambda_t^\ast \lambda_s\lambda_s^\ast =i_{tS}i_{sS}=i_{rS}= \lambda_r\lambda_r^\ast
\]
So multiplying from the left by $\lambda_t^\ast$ and from the right by $\lambda_s$ we get
\[
\lambda_t^\ast \lambda_s=(\lambda_t^\ast \lambda_r)(\lambda_r^\ast \lambda_s)=\lambda_p \lambda_q^\ast
\]

(ii)$\Rightarrow$(i): Let $s,t\in S$ such that $sS\cap tS\neq\varnothing$. Then $\lambda_t^\ast \lambda_s\neq 0$, so there exists $p,q\in S$ with $\lambda_t^\ast \lambda_s=\lambda_p\lambda_q^\ast$. This means that $\lambda_t\lambda_t^\ast \lambda_s\lambda_s^\ast=\lambda_{tp}\lambda_{sq}^\ast$. This element is idempotent, so $tp=sq$ by Lemma \ref{lem:sequalst}. Write $r=tp=sq$. This gives $\lambda_t\lambda_t^\ast \lambda_s\lambda_s^\ast=\lambda_r\lambda_r^\ast$, which is equivalent to $sS\cap tS=rS$.

(iii)$\Rightarrow$(ii) is trivial. It remains to prove (ii)$\Rightarrow$(iii): Let $f\in I_l(S)\setminus\{0\}$. Then there exist $n\in\mathbb{N}$ and $s_1,\ldots, s_n,t_1,\ldots, t_n\in S$ such that $f=\lambda_{t_1}^\ast \lambda_{s_1}\cdots \lambda_{t_n}^\ast \lambda_{s_n}$. By condition (ii) we have that $\lambda_{t_n}^\ast \lambda_{s_n}=\lambda_{p_n}\lambda_{q_n}^\ast$ for some $p_n,q_n\in S$. Assume that for a given $1\leq j\leq n$ there exist $p_j,q_j\in S$ such that
\[
\lambda_{t_j}^\ast \lambda_{s_j}\cdots \lambda_{t_n}^\ast \lambda_{s_n}=\lambda_{p_j}\lambda_{q_j}^\ast
\]
Then
\[
\lambda_{t_{j-1}}^\ast \lambda_{s_{j-1}}\lambda_{t_j}^\ast \lambda_{s_j}\cdots \lambda_{t_n}^\ast \lambda_{s_n}=\lambda_{t_{j-1}}^\ast \lambda_{s_{j-1}}\lambda_{p_j}\lambda_{q_j}^\ast
\]
Now by condition (ii), there exist $p,q\in S$ such that
\[
\lambda_{t_{j-1}}^\ast \lambda_{s_{j-1}}\lambda_{p_j}=\lambda_{t_{j-1}}^\ast \lambda_{s_{j-1}p_j}=\lambda_p\lambda_q^\ast
\]
Setting $p_{j-1}=p$ and $q_{j-1}=qq_j$, we get
\[
\lambda_{t_{j-1}}^\ast \lambda_{s_{j-1}}\lambda_{p_j}\lambda_{q_j}^\ast=\lambda_{p_{j-1}}\lambda_{q_{j-1}}^\ast
\]
By induction on $j$,
\[
\lambda_{t_1}^\ast \lambda_{s_1}\cdots \lambda_{t_n}^\ast \lambda_{s_n}=\lambda_{p_1}\lambda_{q_1}^\ast
\]
This shows that any $f\in I_l(S)\setminus\{0\}$ is on the form $\lambda_p\lambda_q^\ast$ for some $p,q\in S$.
\end{proof}

\begin{corollary}
Suppose $S$ satisfies Clifford's condition. Then $I_l(S)$ is $E^\ast$-unitary if and only if $S$ is cancellative
\end{corollary}
\begin{proof}
One implication was proved in Corollary \ref{cor:eastunitarycancellative}. Suppose $S$ is cancellative, and that $fi_X=i_X$ for some nonzero $f\in I_l(S)$ and some nonempty $X\in J(S)$. Then $f\lambda_r=\lambda_r$ for any $r\in X$ since $X$ is a right ideal. Write $f=\lambda_s\lambda_t^\ast$ with $s,t\in S$. Then $\lambda_t^\ast \lambda_r=\lambda_s^\ast \lambda_r$. Let $p\in\dom{\lambda_t^\ast\lambda_r}$ and define $q=\lambda_t^\ast\lambda_r(p)=\lambda_s^\ast \lambda_r(p)$. Then $tq=sq=rp$. By right cancellativity this gives $t=s$, so $f$ is idempotent.
\end{proof}

\begin{corollary}\label{cor:cliffordhisomorphism}
If $S$ is cancellative and satisfies Clifford's condition, $h:C^\ast_{r,0}(I_l(S))\to C^\ast_r(S)$ is an isomorphism.
\end{corollary}
\begin{proof}
By the previous corollary, $I_l(S)$ is $E^\ast$-unitary. By Proposition \ref{prop:semigrouplattice} $J(S)$ is independent, so Theorem \ref{thm:ucondisomcstar} implies that $h$ is an isomorphism.
\end{proof}

Any semigroup that is the positive cone in a quasilattice ordered group satisfies these conditions. Note however that a semigroup satisfying Clifford's condition is allowed to have nontrivial invertible elements. For instance $(\mathbb{Z}\times\mathbb{Z}^+,+)$ satisfies Clifford's condition, but it is not algebraically ordered.

%% file: lialgebra.tex
\subsection{Li's constructions of full $C^\ast$-algebras for a left cancellative semigroup}

In \cite{li11}, Li defines the full $C^\ast$-algebra $C^\ast(S)$ of a left cancellative semigroup $S$. The construction is as follows: $C^\ast(S)$ is the universal $C^\ast$-algebra generated by isometries $\{v_s:s\in S\}$ and projections $\{e_X:X\in J(S)^0\}$ such that for all $s,t\in S$ and $X,Y\in J(S)^0$,
\begin{align*}
v_{st}=v_sv_t\qquad v_se_Xv_s ^\ast=e_{sX}\\
e_S=1 \qquad e_\varnothing=0 \qquad e_{X\cap Y}=e_Xe_Y
\end{align*}

Li also defines a $C^\ast$-algebra $C^\ast_s(S)$ when $S$ embeds into a group $G$: $C^\ast_s(S)$ is the universal $C^\ast$-algebra generated by isometries $\{v_s:s\in S\}$ and projections $\{e_X:X\in J(S)^0\}$ such that for all $s,t\in S$,
\begin{align*}
v_{st}=v_sv_t\\
e_\varnothing = 0
\end{align*}
and whenever $s_1,\ldots,s_n,t_1,\ldots,t_n\in S$ satisfy $t_1^{-1}s_1\cdots t_n^{-1}s_n=1$ in $G$, then
\[
v_{t_1}^\ast v_{s_1}\cdots v_{t_n}^\ast v_{s_n}=e_X
\]
where $X=t_1^{-1}s_1\cdots t_n^{-1}s_nS$. Li then shows that $\{v_s:s\in S\}\subset C^\ast_s(S)$ and  $\{e_X:X\in J(S)^0\}\subset C^\ast_s(S)$ satisfy the relations defining $C^\ast(S)$, so there exists a surjective $\ast$-homomorphism $\pi_s:C^\ast(S)\to C^\ast_s(S)$ that sends $v_s\in C^\ast(S)$ to $v_s\in C^\ast_s(S)$. The universal property of $C^\ast_s(S)$ also gives a canonical $\ast$-homomorphism $C^\ast_s(S)\to C^\ast_r(S)$ that sends $v_s$ to $V_s$ for all $s\in S$. Moreover, we have:

\begin{proposition}\label{prop:gisomorphism2}
Suppose $S$ embeds into a group $G$. There is a $\ast$-isomorphism $\kappa:C^\ast_s(S)\to C^\ast_0(I_l(S))$ such that $\kappa(v_s)=\lambda_s$ for each $s\in S$.
\end{proposition}
\begin{proof}
The existence of a surjective $\ast$-homomorphism $\kappa:C^\ast_s(S)\to C^\ast_0(I_l(S))$ follows from the universality of $C^\ast_s(S)$. $C^\ast(I_l(S))$ is generated by $\{\lambda_s:s\in S\}$ and $\{i_X:X\in J(S)\}$, and these satisfy the given relations when projected down to $C^\ast_0(I_l(S))$. In particular, if $s_1,\ldots,s_n,t_1,\ldots,t_n\in S$ satisfy $t_1^{-1}s_1\cdots t_n^{-1}s_n=1$ in $G$, then by the proof of Proposition \ref{prop:groupstrongeastunitary}, $f:=\lambda_{t_1}^\ast \lambda_{s_1}\cdots \lambda_{t_n}^\ast \lambda_{s_n}$ is idempotent. So $f=ff^\ast=i_X$ with $X=t_1^{-1}s_1\cdots t_n^{-1}s_nS$.

Let $\mathcal{V}'(S)$ be the subset of $C^\ast_s(S)$ given by
\[
\{v_{t_1}^\ast v_{s_1}\cdots v_{t_n}^\ast v_{s_n}:n\in\mathbb{N},t_1,\ldots, t_n,s_1,\ldots, s_n\in S\}
\]
Li's relations guarantee that $\mathcal{V}'(S)$ is actually an inverse semigroup. Using Lemma 2.8 of \cite{li11} (and mapping down to $C^\ast_s(S)$), we get that for any $v\in \mathcal{V}'(S)$, $v^\ast v=e_X$ and $vv^\ast=e_Y$ for some $X,Y\in J(S)$, so $v$ is a partial isometry. Moreover, any $v,w\in \mathcal{V}'(S)$ have commuting initial and final projections.

Comparing the universal properties of $C^\ast_s(S)$ and $C^\ast_0(\mathcal{V}'(S))$ gives that these two $C^\ast$-algebras are canonically isomorphic. Hence $\kappa$ is an isomorphism if and only if its restriction to $\mathcal{V}'(S)$ gives a semigroup isomorphism $\mathcal{V}'(S)\to I_l(S)$ (note that the restriction of $\kappa$ to $\mathcal{V}'(S)$ is automatically a surjective semigroup homomorphism onto $I_l(S)$). Let $J'(S)$ be the semilattice of idempotents in $\mathcal{V}'(S)$. For any $v\in J'(S)$, $v^\ast v=v$, so $v=e_X$ for some $X\in J(S)$, i.e. $J'(S)=\{e_X:X\in J(S)\}$. Hence $\kappa$ restricts to an isomorphism $J'(S)\to J(S)$ since it is injective on this set.

Let $v=v_{t_1}^\ast v_{s_1}\cdots v_{t_n}^\ast v_{s_n}\in \mathcal{V}'(S)$. Suppose $\kappa(v)=\lambda_{t_1}^\ast \lambda_{s_1}\cdots \lambda_{t_n}^\ast \lambda_{s_n}$ is idempotent. Then $t_1^{-1}s_1\cdots t_n^{-1}s_n=1$ in $G$. This implies that $v$ is idempotent. Thus $\kappa^{-1}(J(S))=J'(S)$, so $\kappa|_{\mathcal{V}'(S)}$ is an isomorphism by Lemma \ref{lem:gisomid}.
\end{proof}

\begin{proposition}\label{prop:gisomorphism}
Let $S$ be any left cancellative semigroup. There is a surjective $\ast$-homomorphism $\eta:C^\ast(S)\to C^\ast_0(I_l(S))$ such that $\eta(v_s)=\lambda_s$ for each $s\in S$. If $S$ satisfies Clifford's condition, then $\eta$ is an isomorphism.
\end{proposition}
\begin{proof}
The existence of $\eta$ follows as before from the universality of $C^\ast(S)$. Define $\mathcal{V}(S)\subset C^\ast(S)$ to be
\[
\{v_{t_1}^\ast v_{s_1}\cdots v_{t_n}^\ast v_{s_n}:n\in\mathbb{N},t_1,\ldots, t_n,s_1,\ldots, s_n\in S\}
\]
As in the previous proposition, $\mathcal{V}(S)$ is an inverse semigroup and it is sufficient to show that the restriction of $\eta$ to $\mathcal{V}(S)$ is a semigroup isomorphism onto $I_l(S)$. Let $J''(S)=\{e_X:X\in J(S)\}$. Then $J''(S)$ is the semilattice of idempotents in $\mathcal{V}(S)$ and $\eta|_{J''(S)}$ is an isomorphism onto $J(S)$.

First we show that $\mathcal{V}(S)\setminus\{0\}=\{v_pv_q^\ast :p,q\in S\}$. The proof is almost identical to that in Proposition \ref{prop:cliffordsimplification}, and we will only show that for any $s,t\in S$ with $v_t^\ast v_s\neq 0$, there are $p,q\in S$ with $v_t^\ast v_s=v_p v_q^\ast$. If $v_t^\ast v_s\neq 0$, then $v_tv_t^\ast v_sv_s^\ast\neq 0$, so since $\eta|_{J''(S)}$ is an isomorphism onto $J(S)$, and since $S$ satisfies Clifford's condition, there is some $r\in S$ with $sS\cap tS=rS$ and
\begin{equation}\label{eq:nicacovariance}
v_tv_t^\ast v_sv_s^\ast=v_rv_r^\ast
\end{equation}
Since $r\in sS\cap tS$, there are $p,q\in S$ such that $sq=r$ and $tp=r$. Then $v_sv_q=v_r$, so $v_q=v_s^\ast v_r$. Similarly, $v_p=v_t^\ast v_r$. Now by multiplying equation \eqref{eq:nicacovariance} on the left with $v_t^\ast$ and on the right with $v_s$, we get that $v_t^\ast v_s=v_t^\ast v_rv_r^\ast v_s=v_pv_q^\ast$.

Consider $v\in \mathcal{V}(S)\setminus\{0\}$, and suppose $\eta(v)$ is idempotent. There are $p,q\in S$ with $v=v_p v_q^\ast$. Now $\eta(v)=\lambda_p \lambda_q^\ast$, so $p=q$ by Lemma \ref{lem:sequalst}. Thus $v=v_pv_q^\ast=v_pv_p^\ast$, which is idempotent. It follows from Lemma \ref{lem:gisomid} that $\eta|_{\mathcal{V}(S)}$ is an isomorphism.
\end{proof}

It is now clear that the canonical map $C^\ast(S)\to C^\ast_r(S)$ factors as
\[
\begin{array}{ccccccc}
C^\ast(S) &\xrightarrow{\eta}& C^\ast_0(I_l(S)) & \xrightarrow{\Lambda_0} & C^\ast_{r,0}(I_l(S)) &\xrightarrow{h} &C^\ast_r(S)
\end{array}
\]
When $S$ embeds into a group we get the following factorization:
\[
\begin{array}{ccccccccc}
C^\ast(S) &\xrightarrow{\pi_s}& C^\ast_s(S) & \xrightarrow[\simeq]{\kappa}& C^\ast_0(I_l(S)) & \xrightarrow{\Lambda_0} & C^\ast_{r,0}(I_l(S)) &\xrightarrow{h} &C^\ast_r(S)
\end{array}
\]
Note that in this case $\eta=\kappa\circ\pi_s$, so $\pi_s$ is an isomorphism if and only if $\eta$ is an isomorphism.

Li asks when a semigroup homomorphism $\phi:S\to R$ of left cancellative semigroups induces a $\ast$-homomorphism $C^\ast(S)\to C^\ast(R)$ by the formula $v_s\mapsto v_{\phi(s)}$. We can give a partial answer. It induces a $\ast$-homomorphism $C^\ast_0(I_l(S))\to C^\ast_0(I_l(R))$ given by $\lambda_s\mapsto \lambda_{\phi(s)}$ if and only if $\phi$ extends to a $0$-homomorphism $I_l(S)^0\to I_l(R)^0$. Of course determining when this is the case may not be easy. See Corollary \ref{cor:functorial} for a result in this direction when $S$ is left reversible and $R$ is a group.

$C^\ast(S)$ has the nice feature that it can be described as a crossed product by endomorphisms (see Lemma 2.14 of \cite{li11}). Li also shows that $C^\ast(S)$ generalizes Nica's $C^\ast$-algebras for quasilattice ordered groups as well as the Toeplitz algebras associated with rings of integers \cite{cuntz_deninger_laca11}. Nica proved in \cite{nica94} that his $C^\ast$-algebra for the quasilattice ordered group $(G,S)$ can be constructed as a $C^\ast$-algebra of the Toeplitz inverse semigroup $\mathcal{T}(G,S)$. This can be explained by the next lemma as well as Proposition \ref{prop:gisomorphism} above. For each $g\in G$, define
\[
\beta_g:\{s\in S:gs\in S\}\to \{s\in S:g^{-1}s\in S\},\qquad \beta_g(s)=gs
\]
$\mathcal{T}(G,S)$ is then defined to be the inverse subsemigroup of $I(S)$ generated by $\{\beta_g\}_{g\in G}$.

\begin{lemma}\label{lem:quasilatticesemigroup}
Let $(G,S)$ be a quasilattice ordered group. Then $I_l(S)^0=\mathcal{T}(G,S)^0$.
\end{lemma}
\begin{proof}
Let $g\in S$. Then $\{s\in S:gs\in S\}=S$ and $\{s\in S:g^{-1}s\in S\}=gS$, so $\beta_g=\lambda_g$. This shows that $\mathcal{T}(G,S)$ contains $I_l(S)$.

Note that for any $g\in G$, $\beta_{g^{-1}}=\beta_g^\ast$. If $\beta_g\neq 0$, $\dom\beta_g^\ast=\{s\in S:g^{-1}s\in S\}=gS\cap S$ is nonempty, so $g^{-1}s=t\in S$ for some $s,t\in S$. Then $g\leq s$ as defined in Definition \ref{def:quasilatticeordered}. Moreover, $1^{-1}s\in S$, so $1\leq s$. Thus $s$ is a common upper bound for $g$ and $1$ in $S$. Then $g$ and $1$ have a least common upper bound $r\in S$ since $(G,S)$ is quasilattice ordered.

Let $p\in gS\cap S$. Using the same arguments as we did for $s$, we get $g\leq p$ and $1\leq p$, so $r\leq p$ since $r$ was a least common upper bound for $g$ and $1$. Then $r\succeq p$, so $p\in rS$. This shows that $gS\cap S\subset rS$. However since $g\leq r$, $g^{-1}r=u$ for some $u\in S$. Then $r=gu$, so $rS=guS\cap S\subset gS\cap S$. Now
\[
\dom\beta_g^\ast=gS\cap S=rS=\dom \lambda_u\lambda_r^\ast
\]
Moreover for any $v\in gS\cap S$,
\[
\beta_g^\ast(v)=\beta_{g^{-1}}(v)=g^{-1}v=ur^{-1}v=\lambda_u\lambda_r^\ast(v)
\]
So $\beta_g^\ast=\lambda_u\lambda_r^\ast$, and $\beta_g=\lambda_r\lambda_u^\ast$. This shows that $\mathcal{T}(G,S)^0\subset I_l(S)^0$.
\end{proof}

A more detailed discussion on the relationship between $I_l(S)$ and $\mathcal{T}(G,S)$ can be found in \cite{jiang03}. Nica's construction of a $C^\ast$-algebra for $\mathcal{T}(G,S)$ uses a groupoid, but Milan explains in section 5 of \cite{milan10} why this $C^\ast$-algebra is isomorphic to $C^\ast_0(\mathcal{T}(G,S))$.

%% file: thicksemi.tex
\subsection{Left reversible semigroups, left amenability and functoriality}\label{sec:leftreversible}

We still consider a left cancellative semigroup $S$ unless something else is stated. Recall that $S$ is \emph{left reversible} if for any $s,t\in S$, $sS\cap tS\neq \varnothing$. The next lemma is a slightly stronger version of Lemma 2.4.8 in \cite{lawson99}.

\begin{lemma}
$S$ is left reversible if and only if $0\notin I_l(S)$ if and only if $\varnothing\notin J(S)$.
\end{lemma}
\begin{proof}
Clearly $0\notin I_l(S)$ if and only if $\varnothing\notin J(S)$. Moreover $sS\cap tS\in J(S)$ for all $s,t\in S$ so $\varnothing\notin J(S)$ implies that $S$ is left reversible. Suppose $S$ is left reversible, and let $X,Y\subset S$ be nonempty right ideals. If $s\in X$ and $t\in Y$, then $sS\subset X$ and $tS\subset Y$, so $sS\cap tS\subset X\cap Y$, and $X\cap Y\neq\varnothing$. Moreover, for any $t\in S$, $tt^{-1}X=X\cap tS$, so $t^{-1}X$ is nonempty. It follows by a simple induction argument that $\varnothing\notin J(S)$.
\end{proof}

We include a short proof of the well known fact that left amenable semigroups are left reversible. See also Proposition (1.23) in \cite{paterson88}. To be formal, a \emph{left invariant mean} on $S$ is a state $\mu$ on $\ell^\infty(S)$ such that for any $s\in S$ and $\xi\in\ell^\infty(S)$, $\mu(\xi\circ\lambda_s)=\mu(\xi)$. $S$ is \emph{left amenable} if it has a left invariant mean. Right amenability is similarly defined. It is not difficult to show that a group or an inverse semigroup is left amenable if and only if it is right amenable, so left amenable groups and inverse semigroups are often just called \emph{amenable}.

\begin{lemma}\label{lem:amenablereversible}
Let $S$ be left amenable. Then for any left invariant mean $\mu$ on $S$, $\mu(\chi_X)=1$ for all $X\in J(S)$. This implies that $S$ is left reversible.
\end{lemma}
\begin{proof}
For convenience we set $\mu(X)=\mu(\chi_X)$ for any $X\subset S$. Note that for any $t\in S$, $\chi_X\circ\lambda_t=\chi_{t^{-1}X}$. Thus $\mu(t^{-1}X)=\mu(X)$. Since $t^{-1}tX=X$, we also have $\mu(tX)=\mu(X)$. By Lemma \ref{lem:simplifiedjs},
\[
J(S)=\{t_1^{-1}s_1\cdots t_n^{-1}s_nS:n\in\mathbb{N},s_i,t_i\in S\}
\]
So by a simple induction argument we get that $\mu(X)=\mu(S)=1$ for all $X\in J(S)$. As $\mu(\varnothing)=0$, this shows that $\varnothing\notin J(S)$.
\end{proof}

It is also well known that a cancellative left reversible semigroup embeds into a group $G$. This is one formulation of Ore's Theorem from \cite{ore31}. The proof we present below in Theorem \ref{thm:orestheorem} is basically the same as Rees' proof that can be found in vol I, p. 35 of \cite{clifford_preston61}. The reason we repeat it here is that it illustrates how $I_l(S)$ is related to $G$. See also ch. 2.4 of \cite{lawson99} where Lawson gives an account of this proof and shows that $I_l(S)$ is $E$-unitary when $S$ is left reversible and cancellative.

Suppose $0\notin I_l(S)$. Then we can construct the maximal group homomorphic image $G(I_l(S))$ of $I_l(S)$ as described in Definition \ref{def:grouphomomorphicimage}. For simplicity we write $G(S)=G(I_l(S))$. Let $\alpha_S:I_l(S)\to G(S)$ denote the quotient homomorphism. Let $\gamma_S:S\to G(S)$ be given by $\gamma_S(s)=\alpha_S(\lambda_s)$. Then $G(S)$ is generated by the cancellative semigroup $\gamma_S(S)$.

\begin{theorem}\label{thm:orestheorem}
Let $S$ be left reversible. The following conditions are equivalent
\begin{enumerate}
\item $S$ is cancellative.
\item $\gamma_S:S\to G(S)$ is injective.
\item $S$ embeds into a group.
\item $I_l(S)$ is $E$-unitary.
\end{enumerate}
\end{theorem}
\begin{proof}
(ii)$\Rightarrow$(iii) and (iii)$\Rightarrow$(i) are trivial. (iii)$\Leftrightarrow$(iv) follows from Proposition \ref{prop:groupstrongeastunitary} ($I_l(S)$ is strongly $E^\ast$-unitary if and only if $S$ is group embeddable) and the fact that an inverse semigroup without $0$ is $E$-unitary if and only if it is strongly $E^\ast$-unitary.

(i)$\Rightarrow$(ii): Since the map $S\hookrightarrow I_l(S)$ is injective, we only have to prove that the homomorphism $I_l(S)\to G(S)$ is injective on the set $\{\lambda_s:s\in S\}$. Let $s,t\in S$ and suppose $\lambda_s$ and $\lambda_t$ map to the same element. Then by the definition of the congruence that was used to construct $G(S)$ there is an $X\subset J(S)$ such that $\lambda_s i_X=\lambda_t i_X$. Hence for any $r\in X$, $sr=\lambda_s(r)=\lambda_t(r)$=tr. By cancelling with $r$ we get that $s=t$.
\end{proof}

\begin{definition}
Let $P$ be semigroup. Recall that a subset $X\subset P$ is said to be \emph{left thick} if for any finite sequence $s_1,\ldots, s_n\in P$,
\[
\bigcap_{i=1}^n s_iX\neq \varnothing
\]
\end{definition}

The next proposition is related to Proposition (1.27) in \cite{paterson88}.

\begin{proposition}
Let $S$ be a subsemigroup of a group $G$ such that $S$ generates $G$. Then $S$ is left reversible if and only if $S$ is a left thick subset of $G$.
\end{proposition}
\begin{proof}
For $t\in S$ and $X\subset S$, $t^{-1}(X)=(t^{-1})X\cap S$ where $t^{-1}(X)$ is the preimage inside $S$ and $(t^{-1})X$ is defined by multiplication inside $G$. So for any $n\in\mathbb{N}$ and $s_i,t_i\in S$,
\[
t_1^{-1}s_1\cdots t_n^{-1}s_nS=S\cap \bigcap_{j=1}^n (t_1^{-1}s_1\cdots t_j^{-1}s_j)S
\]
If $S$ is left thick, this set is never empty, nor is any finite intersection of sets of this type, so $\varnothing\notin J(S)$. On the other hand, for any $g_1\ldots g_m\in G$, write $g_i=t_{i,1}^{-1}s_{i,1}\cdots t_{i,n_i}^{-1}s_{i,n_i}$ with $s_{i,j},t_{i,j}\in S$. Then
\[
\bigcap_{i=1}^m g_iS\supset S\cap\bigcap_{i=1}^m\bigcap_{j=1}^{n_i}(t_{i,1}^{-1}s_{i,1}\cdots t_{i,j}^{-1}s_{i,j})S
\]
If $\varnothing\notin J(S)$, the right hand side is nonempty, and so is the left hand side, so $S$ is a left thick subset of $G$.
\end{proof}

By a theorem of Mitchell \cite{mitchell65}, if $S'$ is a left thick subsemigroup of a of a semigroup $S$, then $S'$ is left amenable if and only if $S$ is left amenable. Hence we get:

\begin{corollary}\label{cor:amenable1}
Let $S$ be cancellative and left reversible. Then $S$ is left amenable if and only if $G(S)$ is amenable.
\end{corollary}

We will now show that the assumption that $S$ is right cancellative is redundant in the statement of Corollary \ref{cor:amenable1}. If $P$ is any semigroup, let $\approx$ (or $\approx_P$) be the relation on $P$ given by $s\approx t$ if there is some $r\in P$ with $sr=tr$. From vol I, p.35 of \cite{clifford_preston61} we have that if $P$ is left reversible, $\approx$ is a congruence and $P/\approx$ is a right cancellative semigroup. Proposition 1.25 of \cite{paterson88} states that $P$ is left amenable if and only if $P/\approx$ is left amenable. Note that $\approx_{I_l(S)}$ is exactly the congruence on $I_l(S)$ one takes the quotient with to create $G(S)$.

\begin{lemma}
Let $S$ be left reversible. Then $\gamma_S(S)$ is isomorphic to $S/\approx_S$
\end{lemma}
\begin{proof}
We show that $\gamma_S(s)=\gamma_S(t)$ if and only if there is some $r\in S$ such that $sr=tr$. First, if $sr=tr$, then $\lambda_s\lambda_r=\lambda_t\lambda_r$, so $\gamma_S(s)=\gamma_S(t)$. On the other hand, if $\gamma_S(s)=\gamma_S(t)$, there is some $X\in J(S)$ such that $\lambda_s i_X=\lambda_t i_X$. Let $r\in X$. Then $\lambda_s i_X \lambda_r=\lambda_s\lambda_r=\lambda_t i_X \lambda_r=\lambda_t\lambda_r$, so by evaluating at $1$ we get $sr=tr$.
\end{proof}

\begin{corollary}\label{cor:amenableequivalence}
Let $S$ be a left cancellative left reversible semigroup. Then $S$ is left amenable if and only if $\gamma_S(S)$ is left amenable if and only if $G(S)$ is amenable if and only if $I_l(S)$ is amenable.
\end{corollary}
\begin{proof}
To show that the amenability of $G(S)$ is equivalent to left amenability of $\gamma_S(S)$, we need to show that $\gamma_S(S)$ is left reversible. We have that for any $s,t\in S$,
\[
\gamma_S(s)\gamma_S(S)\cap\gamma_S(t)\gamma_S(S)=\gamma_S(sS)\cap\gamma_S(tS)\supset\gamma_S(sS\cap tS)
\]
The right hand side is nonempty, so the left hand side must be nonempty as well. This proves that $\gamma_S(S)$ is left reversible since any $p\in\gamma_S(S)$ is on the form $\gamma_S(s)$ for some $s\in S$. All the other equivalences are taken care of by the results we have developed so far: Since $G(S)=I_l(S)/\approx_{I_l(S)}$, $G(S)$ is amenable if and only if $I_l(S)$ is amenable. Since $\gamma_S(S)\simeq S/\approx_S$, $\gamma_S(S)$ is left amenable if and only if $S$ is left amenable.
\end{proof}

We conclude this subsection by showing that when $S$ is left reversible the construction $S\mapsto G(S)$ is a generalization of the Grothendieck construction in that it is functorial. This is probably already known by specialists, but we give a proof here for completeness. Another way to prove it is to show that any homomorphism of $S$ to a group can be extended to define a group homomorphic image of $I_l(S)$, and then use that $G(S)$ is the maximal group homomorphic image of $I_l(S)$.

\begin{lemma}\label{lem:thicksemihomext}
Let $S$ be a subsemigroup of a group $G$ such that $S$ generates $G$, and let $H$ be a group. If $S$ is left thick in $G$, then every homomorphism $\phi:S\to H$ uniquely extends to a homomorphism $\phi':G\to H$.
\end{lemma}
\begin{proof}
Let $t_1,\ldots, t_n,s_1,\ldots, s_n\in S$. We want to define $\phi'(t_1^{-1}s_1\cdots t_n^{-1}s_n)=\phi(t_1)^{-1}\phi(s_1)\cdots \phi(t_n)^{-1}\phi(s_n)$, so we need to show that this is a consistent definition. Let $q_1\ldots q_m,p_1\ldots p_m\in S$ be such that
\[
t_1^{-1}s_1\cdots t_n^{-1}s_n=q_1^{-1}p_1\cdots q_m^{-1}p_m
\]
Since $S$ is left thick in $G$,
\[
S\cap\bigcap_{i=1}^n (s_n^{-1}t_n\cdots s_i^{-1}t_i)S\cap\bigcap_{j=1}^m (p_m^{-1}q_m\cdots p_1^{-1}q_1)S\neq\varnothing
\]
So there exists an $r\in S$ such that
\begin{align*}
u_i&:=t_i^{-1}s_i\cdots t_n^{-1}s_nr\in S\\
v_j&:=q_j^{-1}p_j\cdots q_m^{-1}s_mr\in S
\end{align*}
for all $1\leq i\leq n$ and $1\leq j\leq m$. First, $t_n^{-1}s_nr=u_n$, so $s_nr=t_nu_n$, which implies that $\phi(s_n)\phi(r)=\phi(t_n)\phi(u_n)$ and $\phi(t_n)^{-1}\phi(s_n)\phi(r)=\phi(u_n)$. Suppose now that for some $1\leq k\leq n$,
\[
\phi(t_k)^{-1}\phi(s_k)\cdots \phi(t_n)^{-1}\phi(s_n)\phi(r)=\phi(u_k)
\]
Then since
\[
s_{k-1}t_k^{-1}s_k\cdots t_n^{-1}s_nr=s_{k-1}u_k=t_{k-1}u_{k-1}
\]
we get
\[
\phi(t_{k-1})^{-1}\phi(s_{k-1})\phi(u_k)=\phi(u_{k-1})
\]
Using induction, this implies that $\phi(t_1)^{-1}\phi(s_1)\cdots \phi(t_n)^{-1}\phi(s_n)\phi(r)=\phi(u_1)$. Similarly, $\phi(q_1)^{-1}\phi(p_1)\cdots \phi(q_m)^{-1}\phi(p_m)\phi(r)=\phi(v_1)=\phi(u_1)$, so by cancelling with $\phi(r)$ we see that $\phi'$ is well defined. Uniqueness of $\phi'$ is trivial since $S$ generates $G$.
\end{proof}

\begin{theorem}\label{thm:gofsfunctorial}
Let $S$ be left reversible and let $H$ be a group. Then every homomorphism $\phi:S\to H$ gives rise to a unique homomorphism $\phi':G(S)\to H$ such that $\phi'\circ\gamma_S = \phi$.

Moreover for any left cancellative left reversible semigroup $R$ and homomorphism $\phi:S\to R$, there is a unique homomorphism $\phi':G(S)\to G(R)$ such that $\phi'\circ\gamma_S=\gamma_R\circ \phi$.
\end{theorem}
\begin{proof}
First, we need to show that $\phi:S\to H$ can be pushed down to a homomorphism $\gamma_S(S)\to H$. If $s,t,r\in S$ with $sr=tr$, then $\phi(s)\phi(r)=\phi(t)\phi(r)$, so $\phi(s)=\phi(t)$. This implies that $\phi$ is constant on the equivalence classes of $\approx$, hence there exists a homomorphism $\phi'':\gamma_S(S)\to H$ such that $\phi''\gamma_S=\phi$. By Lemma \ref{lem:thicksemihomext}, $\phi''$ extends to a homomorphism $\phi':G(S)\to H$ such that $\phi'\circ \gamma_S=\phi$. Uniqueness follows since the constructions $\phi\mapsto \phi''$ and $\phi''\mapsto \phi'$ are unique, so if $\psi:S\to G(S)$ is another homomorphism with $\psi\circ\gamma_S=\phi$, then the restriction of $\psi$ to $\gamma_S(S)$ must be equal to $\phi''$.

If $\phi:S\to R$ is a homomorphism, then we can apply the above construction to $\gamma_R\circ \phi:S\to G(R)$ and thereby get the desired $\phi':G(S)\to G(R)$.
\end{proof}

\begin{corollary}\label{cor:functorial}
Let $S$ be left reversible and let $G$ be a group. Then for every homomorphism $\phi:S\to G$ there exists a $\ast$-homomorphism $\pi_\phi:C^\ast(I_l(S))\to C^\ast(G)$ such that $\pi_\phi(\lambda_s)=\lambda_{\phi(s)}$ for each $s\in S$.
\end{corollary}
\begin{proof}
Consider a homomorphism $\phi:S\to G$. From Theorem \ref{thm:gofsfunctorial} there exists a homomorphism $\phi':G(S)\to G$ such that $\phi'\circ\gamma_S = \phi$. Then $\phi'\circ \alpha_S:I_l(S)\to G$ satisfies $\phi'(\alpha_{I_l(S)}(\lambda_s))=\lambda_{\phi(s)}$ for each $s\in S$, so the existence of $\pi_\phi$ follows from the universal property of $C^\ast(I_l(S))$.
\end{proof}

For exampe if $S$ is left reversible we may consider the quotient homomorphism $\alpha_S:I_l(S)\to G(S)$ and obtain a surjective $\ast$-homomorphism $\pi_S:C^\ast(I_l(S))\to C^\ast(G(S))$. (Li also shows the existence of such a map from $C^\ast_s(S)$). When $S=\mathbb{Z}^+$ this is the surjective part of the classical $C^\ast$-extension
\[
0\to K(\ell^2(\mathbb{Z}^+))\to C^\ast_r(\mathbb{Z}^+) \to C(\mathbb{T})\to 0
\]
where $K(\ell^2(\mathbb{Z}^+))$ are the compact operators on $\ell^2(\mathbb{Z}^+)$, $C^\ast_r(\mathbb{Z}^+)\simeq C^\ast_r(I_l(\mathbb{Z}^+))\simeq C^\ast(I_l(\mathbb{Z}^+))$ is the unique $C^\ast$-algebra generated by a single isometry, and $C(\mathbb{T})\simeq C^\ast(\mathbb{Z})\simeq C^\ast(G(\mathbb{Z}^+))$.

By Proposition 1.4 in \cite{duncan_paterson85}, there is actually always a canonical $\ast$-homomorphism $\pi_{S,r}:C^\ast_r(I_l(S))\to C^\ast_r(G(S))$ as well. In general it would be interesting to have a description of the kernel of $\pi_S$ and $\pi_{S,r}$. Nica \cite{nica92} gives some necessary and sufficient conditions for $C^\ast_r(S)$ to contain the compacts when $(G,S)$ is a quasilattice ordered group.

%% file: weakcontainment.tex
\subsection{Amenability and weak containment when $S$ embeds into a group.}\label{sec:weakcont}

In \cite{li11}, Li shows that if $S$ is left reversible, embeds into a group, and $J(S)$ is independent, then $S$ is left amenable if and only if  the canonical map $C^\ast_s(S)\to C^\ast_r(S)$ is an isomorphism. Note that to recover this formulation of the result from from Li's statement, one has to use the fact that when $S$ embeds into a group, $S$ is left reversible if and only if there is a character on $C^\ast_s(S)$. This is proved in Li's Lemma 4.6. One also has to use that since $S$ is left reversible, $S$ is cancellative if and only if $S$ embeds into a group. From Milan's article \cite{milan10}, we know that an $E$-unitary inverse semigroup $P$ has weak containment if and only if $G(P)$ is amenable. Hence Theorem \ref{thm:orestheorem} and Corollary \ref{cor:amenable1} give us:

\begin{theorem}\label{thm:leftreversibleamenable}
A cancellative left reversible semigroup $S$  is left amenable if and only if $I_l(S)$ has weak containment.
\end{theorem}

This lets us recover Li's result.

\begin{corollary}
Suppose $S$ is left reversible, embeds into a group, and $J(S)$ is independent. Then $S$ is left amenable if and only if the canonical map $C^\ast_s(S)\to C^\ast_r(S)$ is an isomorphism.
\end{corollary}
\begin{proof}
Theorem \ref{thm:orestheorem} and Theorem \ref{thm:ucondisomcstar} imply together that $h$ is an isomorphism. Proposition \ref{prop:gisomorphism2} shows that $\kappa$ is an isomorphism. Theorem \ref{thm:leftreversibleamenable} shows that when $S$ is left reversible, $\Lambda$ is an isomorphism if and only if $S$ is left amenable. The composition of $\kappa$, $\Lambda$ and $h$ is the canonical map $C^\ast_s(S)\to C^\ast_r(S)$.
\end{proof}

\begin{corollary}
Suppose $S$ is cancellative, left reversible, and satisfies Clifford's condition. Then the canonical map $C^\ast(S)\to C^\ast_r(S)$ is an isomorphism if and only if $S$ is left amenable.
\end{corollary}
\begin{proof}
By Corollary \ref{lem:amenablereversible}, $h$ is an isomorphism, and by Proposition \ref{prop:gisomorphism}, $\eta$ is an isomorphism. Since $S$ is left reversible, Theorem \ref{thm:leftreversibleamenable} implies that $S$ is left amenable if and only if $\Lambda$ is an isomorphism. The composition of $\eta$, $\Lambda$ and $h$ is the canonical map $C^\ast(S)\to C^\ast_r(S)$.
\end{proof}

\begin{remark}
Corollary \ref{cor:amenableequivalence} does \emph{not} imply that left amenability of $S$ is equivalent to weak containment of $I_l(S)$ for any left reversible $S$. Without $I_l(S)$ being $E$-unitary, one also has to prove that the inverse semigroup $H:=\alpha_S^{-1}(1_{G(S)})$ has weak containment (see Theorem 2.4 and Corollary 2.5 in \cite{milan10}). Milan's results do however give us that weak containment of $I_l(S)$ implies left amenability of $G(S)$ and thus of $S$ (for left reversible $S$).
\end{remark}

When $S$ is not left reversible, $S$ can't be left amenable, but $\Lambda_0:C^\ast_0(I_l(S))\to C^\ast_{r,0}(I_l(S))$ can still be an isomorphism. Nica shows in \cite{nica92} that his version of the full and reduced $C^\ast$-algebras for $\mathbb{F}_n^+$ are canonically isomorphic. Here $\mathbb{F}_n^+$ is the free semigroup on $n$ generators. This implies that $\Lambda_0$ is an isomorphism in this case. $\mathbb{F}_n^+$ is easily seen to be not left reversible for $n>1$.

Milan \cite{milan10} has developed a technique for determining weak containment of strongly $E^\ast$-unitary inverse semigroups $P$. Fixing an idempotent pure grading $\varphi:P\to G^0$, he defines
\begin{align*}
A_g&=\spn\{p\in P:\varphi(p)=g\}\mbox{ inside }\mathbb{C}P/\mathbb{C}0_P\\
B_g&=\overline{A_g}\mbox{ inside }C^\ast_0(P)
\end{align*}
Milan then shows that $\{B_g\}_{g\in G}$ is a Fell bundle over $G$ and that $P$ has weak containment if and only if this Fell bundle is amenable. Milan states this result for the universal grading of $P$, but the proof works for any idempotent pure grading.

In our setting, $I_l(S)$ is strongly $E^\ast$-unitary if and only if $S$ embeds into a group $G$. Recaling the idempotent pure grading $\varphi:I_l(S)^0\to G^0$ constructed in Proposition \ref{prop:groupstrongeastunitary} one sees that the associated Fell bundle $\{B_g\}_{g\in G}$ is given by
\begin{equation}\label{eq:fellbundle}
B_g=\overline{\spn\{ \lambda_{t_1}^\ast \lambda_{s_1}\cdots \lambda_{t_n}^\ast \lambda_{s_n}: t_1^{-1}s_1\cdots t_n^{-1}s_n=g\}}\quad\mbox{ in }C^\ast_0(I_l(S))
\end{equation}

\begin{theorem}
Suppose $S$ embeds into a group $G$. Then $\Lambda_0:C^\ast_0(I_l(S))\to C^\ast_{r,0}(I_l(S))$ is an isomorphism if and only if the Fell bundle $\{B_g\}_{g\in G}$ defined in equation \eqref{eq:fellbundle} is amenable.
\end{theorem}

\begin{corollary}\label{cor:quasilatticeamenable}
Suppose $S$ embeds into a group and satisfies Clifford's condition. Then the canonical map $C^\ast(S)\to C^\ast_r(S)$ is an isomorphism if and only if the Fell bundle $\{B_g\}_{g\in G}$ given by
\begin{equation}\label{eq:fellbundle2}
B_g=\overline{\spn\{ \lambda_s\lambda_t^\ast: st^{-1}=g\}}\quad\mbox{ in }C^\ast_0(I_l(S))
\end{equation}
is amenable.
\end{corollary}
\begin{proof}
For semigroups satisfying Clifford's condition, one can use Proposition \ref{prop:cliffordsimplification} to deduce that the Fell bundles defined in equations \eqref{eq:fellbundle} and \eqref{eq:fellbundle2} are the same.
\end{proof}

When $(G,S)$ is a quasilattice ordered group, this expresses Nica's amenability of $(G,S)$ in terms of the amenability of the Fell bundle defined in eq. \eqref{eq:fellbundle2}, and is in view of Lemma \ref{lem:quasilatticesemigroup} merely a restatement of Proposition 5.2 i \cite{milan10}. In the case where $S$ is a finitely generated free semigroup, amenability of the Fell bundle defined in eq. \eqref{eq:fellbundle2} may be deduced from \cite{exel00}. However the proof one would thereby get from Corollary \ref{cor:quasilatticeamenable} that $C^\ast(S)\to C^\ast_r(S)$ is an isomorphism would not be simpler than Nica's original proof \cite{nica92}.

Li \cite{li11} shows that $C^\ast(S)$ and $C^\ast_r(S)$ are nuclear when $S$ is countable, cancellative and right amenable. The last two conditions imply that $S$ embeds into an amenable group. We show that $S$ does not have to be countable to prove that $C^\ast_s(S)$ is nuclear.

\begin{proposition}\label{prop:amenableisomnuclear}
Suppose $S$ embeds into an amenable group. Then $\Lambda_0$ is an isomorphism, and $C^\ast_0(I_l(S))\simeq C^\ast_s(S)$ and $C^\ast_r(S)$ are nuclear.
\end{proposition}
\begin{proof}
From Theorem 4.7 in \cite{exel97} we know that a Fell bundle over an amenable group satisfies the approximation property, and is thus amenable. Moreover, it was proved in \cite{abadievicens97} that a Fell bundle with nuclear unit fiber has nuclear cross-sectional $C^\ast$-algebra if it also satisfies the approximation property. The unit fiber in $\{B_g\}_{g\in G}$ is the closure of the span of $J(S)$ in $C^\ast_0(I_l(S))$, and is abelian. $C^\ast_r(S)$ is also nuclear since it is a quotient of $C^\ast_0(I_l(S))$ (see Theorem 10.1.3 in \cite{brown_ozawa08}).
\end{proof}

\begin{corollary}
Let $S$ be cancellative and left reversible. Then $S$ is left amenable if and only if $C^\ast_0(I_l(S))\simeq C^\ast_s(S)$ is nuclear. If in addition $J(S)$ is independent, $S$ is left amenable if and only if $C^\ast_r(S)$ is nuclear.
\end{corollary}
\begin{proof}
One implication follows from Proposition \ref{prop:amenableisomnuclear} since a cancellative left amenable $S$ embeds into the amenable group $G(S)$. It remains to see that $C^\ast_0(I_l(S))$ is nuclear implies $S$ is left amenable. This was shown for $C^\ast_s(S)$ in Proposition 4.17 in \cite{li11} with an argument analogous to the following: $C^\ast(G(S))$ is a quotient of $C^\ast_0(I_l(S))$ and is thus nuclear. This implies that $G(S)$ is amenable and that $S$ is left amenable. (See Theorem 10.1.3 and 2.6.8 in \cite{brown_ozawa08}).

Assume that $J(S)$ is independent and $C^\ast_r(S)$ is nuclear. By Theorem \ref{thm:ucondisomcstar} $C^\ast_{0,r}(I_l(S))\simeq C^\ast_r(S)$. $C^\ast_r(G(S))$ is a quotient of $C^\ast_{0,r}(I_l(S))$ and is thus nuclear. This implies that $G(S)$ is amenable (Theorem 2.6.8 in \cite{brown_ozawa08}).
\end{proof}

\begin{remark}
Note that when $S$ is left reversible, but not right cancellative, nuclearity of $C^\ast_0(I_l(S))$ still implies that $S$ is left amenable.
\end{remark}

\begin{corollary}
Let $R$ be a GCD domain. Then the canonical $\ast$-homomorphism $C^\ast(R\rtimes R^\times)\to C^\ast_r(R\rtimes R^\times)$ is an isomorphism and $C^\ast(R\rtimes R^\times)$ is nuclear.
\end{corollary}
\begin{proof}
As Li remarks \cite{li11}, $R\rtimes R^\times$ embeds into an amenable group, so $\Lambda_0$ is an isomorphism and $C^\ast_0(I_l(R\rtimes R^\times))$ is nuclear by Proposition \ref{prop:amenableisomnuclear}. By Corollary \ref{cor:cliffordhisomorphism}, $h$ is an isomorphism, and by Proposition \ref{prop:gisomorphism} and \ref{prop:axpbclifford}, $\eta$ is an isomorphism.
\end{proof}

Note that if one also assumes that $R$ is a Dedekind domain, the previous corollary is weaker than the results given in \cite{cuntz_deninger_laca11,li11} since not all rings of integers or Dedekind domains are GCD domains.

%% file: semigroupcstar.bbl
\begin{thebibliography}{10}

\bibitem{abadievicens97}
Fernando Abadie-Vicens.
\newblock Tensor products of {F}ell bundles over discrete groups.
\newblock {\em arXiv:funct-an/9712006v1}, 1997.

\bibitem{bourbaki72}
Nicolas Bourbaki.
\newblock {\em Commutative Algebra}.
\newblock Hermann, Paris, France, 1972.

\bibitem{brown_ozawa08}
Nathaniel~P. Brown and Narutaka Ozawa.
\newblock {\em ${C}^\ast$-algebras and Finite Dimensional Approximations},
  volume~88 of {\em Graduate Studies in Mathematics}.
\newblock Amer. Math. Soc., Providence, RI, 2008.

\bibitem{chapman_glaz00}
Scott~T. Chapman and Sarah Glaz.
\newblock {\em Non-Noetherian Commutative Ring Theory}.
\newblock Kluwer Academic Publishers, Dordrecht, The Netherlands, 2000.

\bibitem{clifford53}
A.~H. Clifford.
\newblock A class of d-simple semigroups.
\newblock {\em Amer. J. Math.}, 75:547--556, 1953.

\bibitem{clifford_preston61}
A.~H. Clifford and G.~B. Preston.
\newblock {\em The algebraic theory of semigroups}.
\newblock Number~7 in Mathematical Surveys. Amer. Math. Soc., Providence, RI,
  1961.
\newblock Two volumes.

\bibitem{cuntz_deninger_laca11}
Joachim Cuntz, Christopher Deninger, and Marcelo Laca.
\newblock ${C}^\ast$-algebras of {T}oeplitz type associated with algebraic
  number fields.
\newblock {\em arXiv:1105.5352v1}, 2011.

\bibitem{duncan_paterson85}
J.~Duncan and Alan~L.T. Paterson.
\newblock ${C}^\ast$-algebras of inverse semigroups.
\newblock {\em Proc. Edinburgh Math. Soc.}, 28:41--58, 1985.

\bibitem{exel97}
Ruy Exel.
\newblock Amenability for fell bundles.
\newblock {\em Journal fur die reine und angewandte Mathematik}, 1997:41--74,
  1997.

\bibitem{exel00}
Ruy Exel.
\newblock Partial representations and amenable {F}ell bundles over free groups.
\newblock {\em Pacific J. Math}, (192):39--63, 2000.

\bibitem{exel08}
Ruy Exel.
\newblock Inverse semigroups and combinatorial ${C}^\ast$-algebras.
\newblock {\em Bull. Braz. Math. Soc}, 39:191--313, 2008.

\bibitem{exel09}
Ruy Exel.
\newblock Tight representations of semilattices and inverse semigroups.
\newblock {\em Semigroup forum}, 79:159--182, 2009.

\bibitem{hungerford74}
Thomas~W. Hungerford.
\newblock {\em Algebra}.
\newblock Springer Science+Business Media, Inc., New York, NY, 1974.

\bibitem{jiang03}
Zhonghao Jiang.
\newblock The structure of $0$-bisimple strongly ${E}^\ast$-unitary inverse
  monoids.
\newblock {\em Semigroup Forum}, 67:50--62, 2003.

\bibitem{lawson97}
Mark~V. Lawson.
\newblock Constructing inverse semigroups from category actions.
\newblock {\em Journal of pure and applied algebra}, 137:57--101, 1997.

\bibitem{lawson99}
Mark~V. Lawson.
\newblock {\em Inverse semigroups: the theory of partial symmetries}.
\newblock World Scientific Publishing, River Edge, NJ, 1999.

\bibitem{li11}
Xin Li.
\newblock Semigroup ${C}^\ast$-algebras and amenability of semigroups.
\newblock {\em arXiv:1105.5539v2}, 2011.

\bibitem{meakin11}
John Meakin.
\newblock Groups and semigroups: connections and contrasts.
\newblock {\em http://www.math.unl.edu/~jmeakin2/groups and semigroups.pdf},
  2011.
\newblock Submitted for publication.

\bibitem{milan10}
David Milan.
\newblock ${C}^\ast$-algebras of inverse semigroups: amenability and weak
  containment.
\newblock {\em J. Operator Theory}, 63:317--332, 2010.

\bibitem{mitchell65}
Theodore Mitchell.
\newblock Constant functions and left invariant means on semigroups.
\newblock {\em Trans. Am. Math. Soc.}, 119:244--261, 1965.

\bibitem{murphy96}
G.~J. Murphy.
\newblock ${C}^\ast$-algebras generated by commuting isometries.
\newblock {\em Rocky Mountain Journal of Mathematics}, 26:237--267, 1996.

\bibitem{nica92}
A.~Nica.
\newblock ${C}^\ast$-algebras generated by isometries and {W}iener-{H}opf
  operators.
\newblock {\em J. Operator Theory}, 27:17--52, 1992.

\bibitem{nica94}
A.~Nica.
\newblock On a groupoid construction for actions of certain inverse semigroups.
\newblock {\em Internat. J. Math.}, 5:349--372, 1994.

\bibitem{ore31}
{\O}ystein Ore.
\newblock Linear equations in non-commutative fields.
\newblock {\em Ann. of Math.}, 32:463--477, 1931.

\bibitem{paterson78}
Alan L.~T. Paterson.
\newblock Weak containment and {C}lifford semigroups.
\newblock {\em Proc. Roy. Soc. Edinburgh}, 81A:23--30, 1978.

\bibitem{paterson88}
Alan L.~T. Paterson.
\newblock {\em Amenability}.
\newblock Number~29 in Mathematical Surveys and Monographs. Amer. Math. Soc.,
  Providence, RI, 1988.

\bibitem{paterson99}
Alan L.~T. Paterson.
\newblock {\em Groupoids, Inverse Semigroups, and their Operator algebras},
  volume 170 of {\em Progress in Mathematics}.
\newblock Birkhauser, Boston, MA, 1999.

\end{thebibliography}
